\documentclass[11pt]{article}
\usepackage[T1]{fontenc}
\usepackage{amssymb,amsmath,amsthm}
\usepackage{graphicx}
\usepackage{indentfirst}
\usepackage[numbers,sort&compress]{natbib}
\usepackage{authblk}
\usepackage{float}
\usepackage{latexsym,amsfonts,graphics,mathtools}
\usepackage{abstract}
\usepackage{siunitx}
\usepackage{subcaption}
\usepackage{geometry}
\usepackage{overpic}
\usepackage[french,english]{babel}
\usepackage{color}
\usepackage[colorlinks,linkcolor=blue]{hyperref}
\usepackage{tikz}

\usetikzlibrary{decorations.pathreplacing}
\usetikzlibrary{positioning} %为了实现相对位置的设定
\usepackage{xcolor}
\usepackage{environ}
\usepackage{enumitem}
\numberwithin{equation}{section}
\newtheorem{theorem}{Theorem}[section]
\newtheorem{lemma}[theorem]{Lemma}
\newtheorem{corollary}[theorem]{Corollary}
\newtheorem{proposition}[theorem]{Proposition}
\newtheorem{definition}[theorem]{Definition}

\newtheorem{remark}{Remark}
\makeatletter
\newsavebox{\measure@tikzpicture}
\NewEnviron{scaletikzpicturetowidth}[1]{%
  \def\tikz@width{#1}%
  \def\tikzscale{1}\begin{lrbox}{\measure@tikzpicture}%
  \BODY
  \end{lrbox}%
  \pgfmathparse{#1/\wd\measure@tikzpicture}%
  \edef\tikzscale{\pgfmathresult}%
  \BODY
}
\makeatother

\title{Multivariate multifractal analysis of L\'evy functions. Part I: Determination of multifractal spectra}
\author[1]{St\'ephane Jaffard \thanks{stephane.jaffard@u-pec.fr}}
\author[2]{Lingmin Liao \thanks{lmliao@whu.edu.cn}}
\author[1]{Qian Zhang \thanks{qian.zhang@u-pec.fr}}
\affil[1]{{\small{Université Paris Est Créteil, Université Gustave Eiffel, CNRS, LAMA UMR8050, F-94010 Créteil, France}}}
\affil[2]{{\small{School of Mathematics and Statistics, Wuhan University, Wuhan, China}}}
\begin{document}
\maketitle
\begin{abstract}
We study the sets of points where a L\'evy function and a translated L\'evy function share a given couple of   H\"older exponents, and we investigate how  their Hausdorff dimensions   depend on the translation parameter.
\end{abstract}

%%%%%%%%%%%%%%%%%%%%%%%%%%%%%%%%%%%%%%%%%%%%%%%%%%%%%%%%%%%%%%%%%%%%%%%%%%%%%%%%%%%%%%%%%%%%%%%%%%%%%%%%%%%%%%%%%%%%
\section{Introduction}
%%%%%%%%%%%%%%%%%%%%%%%%%%%%%%%%%%%%%%%%%%%%%%%%%%%%%%%%%%%%%%%%%%%%%%%%%%%%%%%%%%%%%%%%%%%%%%%%%%%%%%%%%%%%%%%%%%%%

The multifractal analysis of a function $f$ consists of estimating the size of sets of points which share a same pointwise regularity exponent of $f$. %Let $f:\mathbb{R}^d\to\mathbb{R}^d$ be a function or a measure.
To measure the pointwise regularity of a function, the H\"older exponent is usually used, see \cite{jaffard2004wavelet}; note, however, that other pointwise regularity exponents, such as the $p$-exponents,  and the weak scaling exponent, have also been considered, see  \cite{jaffard2005wavelet,abry2015bridge,abry2015irregularities,jaffard2016p,billat2023multifractal,leonarduzzi2017finite}. 
%We will come back to the question of the choice of pointwise regularity exponents  in the conclusion. 

\begin{definition}\label{def}
Let $f : \mathbb{R^d}\rightarrow \mathbb{R}$ be a locally bounded function, $\alpha >0$, and $x_0\in \mathbb{R^d}$. The function  $f$ belongs to $C^{\alpha}(x_0)$ if there exist $C > 0$, $r >0$ and a polynomial $P$ satisfying $deg(P) < \alpha$ such that
\begin{equation}\label{Hol} \nonumber
\forall x \in[x_0-r,x_0+r],\ \lvert f(x)-P(x-x_0)\lvert  \leqslant C|x-x_0|^{\alpha}.
\end{equation}

The \textit{H\"older exponent} of $f$ at $x_0$ is 
\begin{equation} \nonumber
\begin{aligned}
h_f(x_0)=\sup\{\alpha \geqslant 0:\ f\in\ C^{\alpha}(x_0)\}.
\end{aligned}
\end{equation}
\end{definition}

Let  $E_f(H)$ be the set of points which share the same pointwise regularity exponent $H$:
\[ E_f(H) = \{ x: \quad h_f(x) =H\}. \]
In many mathematical examples, $E_f(H)$ is a fractal set.
The relevant information about  $f$ is not given directly by the pointwise regularity exponents, but it is encoded through the size of the level sets $E_f(H)$. To encapsulate such information, one determines the { \em multifractal spectrum}  of $f$ 
\begin{equation} \nonumber
\mathcal{D}_f:H\longmapsto \dim_H\big(E_f(H)\big),
\end{equation}
where $\dim_H$ stands for the Hausdorff dimension.
By convention, $\dim_H(\emptyset)=-\infty$,  and the support of the spectrum is the set of values $H$ such that $E_f(H)\neq \emptyset$.

The {\em multifractal formalism} is a numerically robust way of estimating the multifractal spectrum; it yields a collection of parameters which can be used in practice for classification or model selection. These notions have been widely applied in many fields and have reached great successes, see e.g.  \cite{jaffard2004wavelet,liao2014multifractal,abry2015irregularities,jaffard2016p,leonarduzzi2016p,fan2018multifractal}. However,  these achievements have been obtained by analyzing one function at a time; that is, it is implemented in a univariate environment. Nowadays,  many systems record a large number of data and several signals need to be processed simultaneously; these new needs call for a theoretical study of {\em multivariate multifractal analysis}  that would take into account the possible interrelationships between several functions, and would put into light key information concerning the correlations of their singularity sets. 
{\em Multivariate multifractal analysis}  is a simultaneous multifractal analysis of several pointwise regularity exponents derived from one or several functions. 

Let $(f_1,\cdots,f_d)$ be  $d$ functions (which may coincide or differ) to which are associated pointwise regularity exponents $(h_{1,f_1}, h_{2,f_2},\cdots, h_{d,f_d})$, given $H_1, H_2, \cdots, H_d >0$, one is  interested in the level sets
\begin{equation}
\begin{aligned}
&E_{f_1,f_2,\cdots,f_d}(H_1, H_2, \cdots, H_d) \\
:=&  \big\{x \in \mathbb{R}: h_{1,f_1}(x) = H_1, h_{2,f_2}(x) = H_2, \cdots, h_{d,f_d}(x) = H_d\big\}
\\=&\ \big\{x \in \mathbb{R}: h_{1,f_1}(x) = H_1\big\}\cap\big\{x \in \mathbb{R}: h_{2,f_2}(x) = H_2\big\}\cap\cdots \cap\big\{x \in \mathbb{R}: h_{d,f_d}(x) = H_d\big\}
\\=&\ \bigcap_{i=1}^dE_{f_i}(H_i). \label{multiset}
\end{aligned}
\end{equation}

\begin{definition}
The multivariate multifractal spectrum  of $(f_1,\cdots,f_d)$ is 
\begin{equation} \nonumber
\mathcal{D}_{f_1,f_2,\cdots,f_d} (H_1,H_2,\cdots,H_d)=  \dim_H\left(E_{f_1,f_2,\cdots,f_d}(H_1, H_2, \cdots, H_d)\right).\label{mspec}
\end{equation}
Its support  consists of the sets of exponents $(H_1,\cdots,H_d)$, such that $E_{f_1,\cdots,f_d}(H_1, \cdots, H_d) \neq \emptyset$.
\end{definition}

Multivariate multifractal analysis was first introduced in \cite{meneveau1990joint} as a natural extension of univariate analysis in order to analyze jointly several related quantities in hydrodynamic turbulence; in this article, a {\em multivariate multifractal formalism} was proposed, the purpose of which is to supply a numerically stable method in order to estimate the multivariate multifractal spectrum.     
%Since then, the multivariate multifractal analysis was theoretically investigated in \cite{barreira2002higher,peyriere2004vectorial,olsen2005mixed,jaffard2019multivariate}, and was applied to practical situations in \cite{wendt2018assessing,leonarduzzi2018multifractal,jaffard2019multifractal}.
The properties of univariate multifractal spectra have been extensively studied, but those of multivariate multifractal spectra have been the object of few mathematical investigations. J. Peyri\`ere \cite{peyriere2004vectorial} developed a formalism to perform the multifractal analysis of vector valued functions in a metric space. L. Barreira, B. Saussol, and J. Schmeling \cite{barreira2002higher} established a higher-dimensional version of multifractal analysis for several classes of hyperbolic dynamical systems. L. Olsen \cite{olsen2005mixed} investigated such spectra for several self-similar measures. A wavelet based formalism was proposed by M. Ben Slimane, A. Ben Mabrouk,  and J.  Aouidi in \cite{ben2016mixed} and its validity was investigated for specific functions;     in  \cite{ben2017mixed},  M. Ben Slimane, I. Ben Omrane, and B. Halouani proved the multifractal formalism holds for a residual (topologically large) subset of the product Besov space. 
% { \color{blue} la formule generique obtenue est-elle la formule de grande intersection? si oui, le signaler plus loin, quand on parle de cette formule---yes, it's grande intersection and I cited in the following}.   
 In  \cite{jaffard2019multifractal,jaffard2019multivariate},  
S. Jaffard, S. Seuret, H. Wendt, R. Leonarduzzi, S. Roux,
and P. Abry performed the joint multifractal analysis of a collection of signals unraveling the correlations between the locations of their pointwise singularities. They provided a framework which allows   to estimate the  multivariate multifractal spectrum and worked out some examples illustrating how the correlations between mathematical models are reflected in their multivariate multifractal spectrum. They also discussed the relationship between the multivariate multifractal spectrum and the corresponding Legendre spectrum.

According to \eqref{multiset}, obtaining a multivariate multifractal spectrum is equivalent to calculating the Hausdorff dimension of the intersection of level sets of regularity exponents.  In order to simplify notations, we now consider  two sets only. Many mathematical results have been derived concerning  the dimension of the intersection of sets and may provide relevant guidance for multivariate multifractal analysis, e.g. \cite{mattila1999geometry,falconer2004fractal}.
The first general result  follows from the classical rule for the intersections of subspaces: In $\mathbb{R}^d$, affine subspaces generically intersect according to the "sum of codimension rule":
\begin{equation}\label{qwe} \nonumber
\dim_H(A \cap B)= \dim_H(A) + \dim_H(B) - d
\end{equation}
(if the quantity on the right hand side is positive). 
P. Mattila \cite{mattila1985hausdorff} proved the following result (see also \cite[p.179]{mattila1999geometry}).

\begin{theorem}
    Let $O(n)$ denote the orthogonal group of $\mathbb{R}^d$ and $\theta_d$ be its Haar probability measure. Let $\tau_x$ be the translation defined by $\tau_x(y)=x+y$ for any $y\in\mathbb{R}$. For $d\geqslant 2$, and $s,\ t >0$ such that $s+t\geqslant d$, $t>(d+1)/2$, if $A$ is $\mathcal{H}^s$ measurable with $\mathcal{H}^s(A)<\infty$ and $B$ is $\mathcal{H}^t$ measurable with $\mathcal{H}^t(B)<\infty$, then for $\mathcal{H}^s\times \mathcal{H}^t\times \theta_d$ almost all $(x,y,g)\in A\times B\times O(n)$,
\begin{equation} \nonumber
	\dim_H\big(A\cap(\tau_x\circ g\circ\tau_{-y})B\big)\geqslant s+ t - d.
\end{equation}
\end{theorem}

 Thus, a first possible  generic formulation for bivariate multifractal spectrum that can be expected is
\begin{equation} \label{cod}
D(H_1, H_2) = D(H_1) + D(H_2) - d.
\end{equation}
However, large classes of fractal sets satisfy the {\em large intersection property}, i.e. 
\begin{equation} \label{formulemin} 
\dim_H(A \cap B) = \min\big\{\dim_H A,\ \dim _HB\big\}.
\end{equation}
Indeed, though such a formula may seem counter-intuitive, it has been shown to hold in general frameworks  where the sets $A$ and $B$ can be written as limsup sets, see e.g. \cite{falconer2004fractal,durand2008sets,falconer1994sets} and references therein. This alternative formula is of particular relevance here in the  setting of multifractal analysis because the sets   \[ G_f(H)=\{x\in\mathbb{R}:h_f(x)\leqslant H\}\]   often are of this type. It is  e.g. the case for random wavelet series, L\'evy processes, Davenport series, the approximation by rational numbers of elements in the middle third Cantor set, see \cite{aubry2002random,jaffard1999multifractal,jaffard2000lacunary,daviaud2022heterogeneous} to mention but a few (indeed, in such cases the sets $G_f(H)$  are limsup sets). In \cite{ben2017mixed}, the generic formula obtained corresponds to the "large intersection" result in the framework of Baire category theory. This leads to a second possible general formulation for the multivariate multifractal spectrum :
\begin{equation}\label{minv}
D(H_1, H_2) = \min\big\{D(H_1), D(H_2)\big\}
\end{equation}
(note, however, that, in contrast with the sets $G_f (H)$, the sets $E_f (H)$ { \em  are not } limsup sets).  

It is still  unclear to determine the respective ranges of validity of \eqref{cod} and \eqref{minv}. Actually, very few mathematical examples have been considered and they are too scarce to yield a reliable intuition on this topic.
Therefore, obtaining explicit multivariate multifractal spectra for a wide range of examples is an important open question. In this paper, we determine explicitly multivariate multifractal spectra in some specific cases where the fractal sets involved are supplied by $b$-adic approximation properties. We will consider couples of shifted {\em L\'evy functions}, see  Definition \ref{deflev} below.  One motivation for focusing on the implications of a shift on bivariate  multifractal analysis comes from applications: data are often recorded with possible shifts the value of which is unknown; an example is supplied by EEG signals recorded at different locations, which involve  time lags between the reception of information coming from  different areas  of the brain. Therefore, the  theoretical understanding  of the implications of time shifts on bivariate multifractal analysis is an important issue.  

We divide the multivariate multifractal analysis of L\'evy functions into two parts. In the present article (Part I), we estimate the multifractal spectrum of L\'evy functions. In Part II,  we will study the validity of the multivariate multifractal formalism, see \cite{Qianpart2}.

% { \color{blue} je ne comprends pas cette remarque: on a ici des blocs de zeros de plus en plus longs, ce qui est exclu dans le theoreme \ref{prop2}.---------I just want to show that we can directly get the bivariate spectrum of $E^{\leqslant}$ based on the large intersection property, but not the set $E^=$, maybe I delete this paragraph?} {\color{red} Put it later, we have not defined the functions yet.}

\smallskip
Our paper is organized as follows. In Section \ref{section2.1}, we define L\'evy functions, followed by the presentation of our main results on the bivariate multifractal spectra of L\'evy functions in Section \ref{section2.2}. Section \ref{2} covers the necessary preliminaries for conducting multifractal analysis;  specifically, Section \ref{21} explores the relationship between the H\"older exponent of L\'evy functions and Diophantine approximation, while Section \ref{subsection2.2} introduces the $b$-adic approximation and symbolic space, both crucial for determining the spectra of L\'evy functions.  Section \ref{3} presents the proof of Theorem \ref{th}, which describes the part of the bivariate  spectrum  of  $L_{\alpha_1}^b$ and the shifted function  $L_{\alpha_2}^{b,y}$ that depends solely on the $b$-adic approximation properties of the translation parameter $y$. In Section \ref{section8}, we  prove  Theorem \ref{greenpartspectrum} which yields that the values taken by the  bivariate level sets  are either  empty  or satisfy the large intersection formula. Finally,  Section \ref{section7} presents the proofs of Propositions \ref{prop3} and \ref{prop1}, which provides examples demonstrating that the spectra depend on the explicit $b$-ary approximation properties of the translation $y$ (and not only on its $b$-adic approximation exponent.

%  provide proofs of the theorems and propositions related to the bivariate multifractal spectrum.

% In Section \ref{section2}, we present our main results of the bivariate multifractal spectrum of L\'evy functions. In Section \ref{2}, we introduce the $b$-adic approximation and symbolic space which are important for determining the spectra of L\'evy functions. Then in Sections \ref{3} - \ref{section8}, we give the proofs of the theorems and propositions concerning the bivariate multifractal spectrum.

%%%%%%%%%%%%%%%%%%%%%%%%%%%%%%%%%%%%%%%%%%%%%%%%%%%%%%%%%%%%%%%%%%%%%%%%%%%%%%%%%%%%%%%%%%%%%%%%%%%%%%%%%%%%%%%%%%%%
\section{Statement of the main results}\label{section2}
%%%%%%%%%%%%%%%%%%%%%%%%%%%%%%%%%%%%%%%%%%%%%%%%%%%%%%%%%%%%%%%%%%%%%%%%%%%%%%%%%%%%%%%%%%%%%%%%%%%%%%%%%%%%%%%%%%%%

%%%%%%%%%%%%%%%%%%%%%%%%%%%%%%%%%%%%%%%%%%%%%%%%%%%%%%%%%%%%%%%%%%%%%%%%%%%%%%%%%%%%%%%%%%%%%%%%%%%%%%%%%%%%%%%%%%%%
\subsection{ L\'evy functions} \label{section2.1}
%%%%%%%%%%%%%%%%%%%%%%%%%%%%%%%%%%%%%%%%%%%%%%%%%%%%%%%%%%%%%%%%%%%%%%%%%%%%%%%%%%%%%%%%%%%%%%%%%%%%%%%%%%%%%%%%%%%%

The $1$-periodic 'saw-tooth' function is 
\begin{equation}\nonumber
\begin{aligned}
\{x\}=
\left\{
             \begin{array}{lr}
            x-\lfloor x \rfloor - \frac{1}{2},&\quad \text{if}\ x\ \notin\ \mathbb{Z},\\
            0,&\quad \text{if}\ x\ \in\ \mathbb{Z},
             \end{array}
\right.
\end{aligned}
\end{equation}
where $\lfloor x \rfloor$ denotes the integer part of the real number $x$.

\begin{definition} \label{deflev} Let 
$\alpha > 0$ and $b \in \mathbb{N}$ with $ b \geqslant 2$; the L\'evy function $L_{\alpha}^b$ is 
\begin{equation} \label{lef} 
L_{\alpha}^b(x)=\sum_{i=1}^{\infty}\frac{\{b^ix\}}{b^{\alpha i}},\quad \forall x \in \mathbb{R}.
\end{equation}
\end{definition} 

L\'evy functions were introduced by Paul L\'evy \cite[Chapter 5]{PaulLevy} as  toy examples of deterministic functions which
exhibit some characteristics of L\'evy processes. They provide simple examples of function  with a dense set of discontinuities, but in the form of superpositions of "compensated jumps": the compensation is performed through the linear part of $\{b^ix\}$  (and this compensation  is required in order to  make the series converge if $0< \alpha <1$). Note that if $\alpha <0$ then the series \eqref{lef}  is convergent in the sense of distributions but its sum is no longer a locally bounded function \cite{jaffard1997old,jaffard2004davenport}. 

Since $|\{x\}|\leqslant $1, the series \eqref{lef} is normally convergent toward a $1$-periodic function. Since the function $\{b^ix\}$ is continuous except at the $b$-adic rational  points $k \cdot b^{-i}$;  it follows that  $L_\alpha^{b}$ is also continuous except at $b$-adic rational points where it has a jump of amplitude $C\cdot b^{-\alpha i} $. 

 L\'evy functions are simple examples of  { \em Davenport series}, which constitute an important class of ``pure jump functions''  that exhibit multifractal behavior \cite{jaffard2004davenport}.

% { \color{blue} donner un petit historique sur les fonctiosn de L\'evy: pourquoi elles ont ete introduites?  dire aussi qu'elles rentrent dans le cadre plus general des series de Davenport qui ont fait l'objet de nombreuses etudes concernant leur aspects multifractals}

Let $y \in\mathbb{R}$; the translated L\'evy function  $L_\alpha^{b,y} $ is 
\begin{equation} 
 \forall x \in \mathbb{R} \qquad L_\alpha^{b,y}(x)=\sum_{i=1}^{\infty}\frac{\{b^i(x-y)\}}{b^{\alpha i}}.\label{lef2}
\end{equation}

 Denote by $h_{L_{\alpha_1}^b}(x)$ and $h_{L_{\alpha_2}^{b,y}}(x)$ the H\"older exponents of  $L_{\alpha_1}^b$ and $L_{\alpha_2}^{b,y}$ at $x$ respectively.
%For L\'evy function $L_{\alpha_1}^b$ and the translated L\'evy function $L_{\alpha_2}^{b,y}$,
Our purpose in this article is to study the level sets 
\begin{equation} \label{fin} \nonumber
E_{L_{\alpha_1}^b,L_{\alpha_2}^{b,y}}(H_1,H_2)=\left\{x\in\mathbb{R}: h_{L_{\alpha_1}^b}(x)=H_1, h_{L_{\alpha_2}^{b,y}}(x)=H_2\right\},\ (H_1, H_2 > 0),
\end{equation}
in order to determine  the bivariate multifractal spectrum
\begin{equation} \nonumber
\mathcal{D}_{L_{\alpha_1}^b,L_{\alpha_2}^{b,y}}:(H_1, H_2)\longmapsto \dim_H\left(E_{L_{\alpha_1}^b,L_{\alpha_2}^{b,y}}(H_1,H_2)\right).
\end{equation}

First, remark that  $\{ x \} = -\{ 1-x \} $ so that  $L_{\alpha}^b (x) = -L_{\alpha}^b (1-x) $ and the operation $x \rightarrow 1-x$ changes $\varepsilon_i $ into $\varepsilon_{b-i} $ in the $b$-ary expansion of $x$ and $1-x$ respectively (for the specific definition, see Subsection \ref{subsection2.2}). Then it follows that the result on the bivariate multifractal spectrum of $L_{\alpha_1}^b$ and $L_{\alpha_2}^{b,y}$ does  not change if we replace the shift $y$ of $L_{\alpha_2}^{b,y}$ by $1-y$.

\smallskip
On the other hand, we remark that for the lower level sets of L\'evy functions
\begin{equation} \nonumber
\begin{aligned}
E_{L_{{\alpha_1}}^b,L_{\alpha_2}^{b,y}}^{\leqslant}(H_1,H_2)&:=\left\{x \in \mathbb{R}: h_{L_{\alpha_1}^b}(x) \leqslant H_1, h_{L_{\alpha_2}^{b,y}}(x) \leqslant H_2\right\}\\
&=\left\{x \in \mathbb{R}: h_{L_{\alpha_1}^b}(x) \leqslant H_1\right\} \bigcap \left\{x \in \mathbb{R}: h_{L_{\alpha_2}^{b,y}}(x) \leqslant H_2\right\}
\\&=:E_{L_{{\alpha_1}}^b}^{\leqslant}(H_1)\cap E_{L_{\alpha_2}^{b,y}}^{\leqslant}(H_2),\label{uisoset}
\end{aligned}
\end{equation}
the whole multifractal spectrum is easy to obtain. Indeed, the following result of  \cite[p.193]{borosh1972generalization}, states that, for $H\in[0, {\alpha_1}],$
\begin{equation} \nonumber
\dim_H\Big(E_{L_{{\alpha_1}}^{b,y}}^{\leqslant}(H)\Big)=\dim_H\Big(E_{L_{{\alpha_1}}^b}^{\leqslant}(H)\Big)=\dim_H\Big(E_{L_{{\alpha_1}}^b}(H)\Big)= \frac{H}{{\alpha_1}}.
\end{equation}
Further, the sets $E_{L_{{\alpha_1}}^b}^{\leqslant}(H_1)$ and $E_{L_{\alpha_2}^{b,y}}^{\leqslant}(H_2)$ satisfy the large intersection property (\cite[p.279]{falconer1994sets}). Thus, for all $(H_1,H_2) \in [0,{\alpha_1}] \times [0, {\alpha_2}]$,
\begin{equation} \nonumber
\begin{aligned}
\dim_H\left(E_{L_{{\alpha_1}}^b,L_{\alpha_2}^{b,y}}^{\leqslant}(H_1,H_2)\right)&=\min\left\{\dim_H\left(E_{L_{{\alpha_1}}^b}^{\leqslant}(H_1)\right),\dim_H\left(E_{L_{\alpha_2}^{b,y}}^{\leqslant}(H_2)\right)\right\}
\\&=\min\left\{\frac{H_1}{{\alpha_1}},\frac{H_2}{{\alpha_2}}\right\}.
\end{aligned}
\end{equation}

However, observe that for the H\"older exponent $h_{L_{\alpha_1}^b}$, if $H_1\neq H_2$, then
\begin{equation} \nonumber
E_{L_{{\alpha_1}}^b}(H_1) \cap E_{L_{{\alpha_2}}^b}(H_2)=\left\{x\in \mathbb{R}:h_{L_{\alpha_1}^b}(x)=H_1\right\}\cap\big\{x\in \mathbb{R}:h_{L_{\alpha_1}^b}(x)=H_2\big\}= \emptyset.
\end{equation}
Hence, the sets $E_{L_{{\alpha_1}}^b}(H_1)$ and $E_{L_{\alpha_2}^{b,y}}(H_2)$ do not satisfy the large intersection property.
So the bivariate multifractal spectrum $\mathcal{D}_{L_{{\alpha_1}}^b, L_{\alpha_2}^{b,y}}$ can not be directly obtained as a consequence of the large intersection property.
%%%%%%%%%%%%%%%%%%%%%%%%%%%%%%%%%%%%%%%%%%%%%%%%%%%%%%%%%%%%
\subsection{Bivariate multifractal spectrum} \label{section2.2}
%%%%%%%%%%%%%%%%%%%%%%%%%%%%%%%%%%%%%%%%%%%%%%%%%%%%%%%%%%%%

%By using the techniques of combinatorics on words, we will prove the following main theorem which yields the part of the spectrum which depends only on the 
In order to state the main results of our paper, we need to recall the notion of 
{\em $b$-adic approximation}.

\begin{definition} \label{def2}
Let $\alpha \geqslant 1$. A real number $x$ is $\alpha$-approximable by b-adic rationals if there exist infinitely many $n$ such that
\begin{equation} \nonumber
\nonumber
\lVert b^{n}x \rVert \leqslant \frac{1}{b^{n(\alpha-1)}},
\end{equation}
where $\lVert\cdot\rVert$ denotes the distance to the nearest integer. For $x\in \mathbb{R}$, its $b$-adic approximation exponent is defined by
\begin{equation}  \label{deltab} 
\Delta^{b}(x)=\sup\left\{\alpha\geqslant1: \lVert b^nx \rVert \leqslant \frac{1}{b^{n(\alpha-1)}} {\rm\ for\ infinitely \ many}\ n\right\}.
\end{equation}
\end{definition}

For $\alpha_1, \alpha_2\geqslant1$ and $y\in \mathbb{R}$, denote
\[K_{ \alpha_1,\alpha_2 }(y)=  \left[0,\frac{{\alpha_1}}{\Delta ^b(y)}\right]\times \left[0,\frac{{\alpha_2}}{\Delta ^b(y)}\right] \cup \left\{(H_1, H_2) : \  H_1\in [0,{\alpha_1}] \mbox{ and } H_2=\frac{{\alpha_2}}{{\alpha_1}} H_1 \right\} \]
(see Figure \ref{fig:M9}). 

% {\color{blue}can delete?  By Definition \ref{def2}, naturally,

% { \color{blue} il y a un probleme: la formule  en dessous ne definit pas un exposant exact mais les points  approxim\'es  mieux qu'a une certaine vitesse.  

% Du coup, on ne sait pas quel est le $  A_{\eta}$ dans ce qui suit }
% By Definition \ref{def2}, naturally,
% \begin{equation}
%     \begin{aligned}
%         \nonumber
%         A_{\eta} \subset \bigcap\limits_{n=1}^{\infty}\bigcup\limits_{j=n}^{\infty}\bigcup\limits_{p=0}^{b^j-1}\left[\frac{p}{b^j}-\frac{1}{b^{j\eta}}, \frac{p}{b^j}+\frac{1}{b^{j\eta}} \right]. 
%     \end{aligned}
% \end{equation} 

% c'est répété après (fin de la remarque 1) et  un ensemble défini par un exposant d'approximation n'est pas inclus dans dl'ensemble avec le meme ordre exact d'approximation, car on peut permettre une perte logarithmique (par exemple) }

% { \color{blue} je vois toujours un probleme: si on note $E_\eta$ l'ensemble de droite, on peut seulement dire que 
% \[ A_\eta \subset\bigcap_{ \eta' < \eta } E_{\eta' }\] qui est plus gros que $E_\eta$! il faut corriger et surtout verifier que ca ne pose pas des problemes par la suite---------I didn't use the previous notion that I commented it out now, the explanation is in Remark 1}

The following theorem will be proven in Section \ref{3}.
\begin{theorem}\label{th}
The bivariate multifractal spectrum  of the L\'evy function $L_{\alpha_1}^b$ and the translated  function $L_{\alpha_2}^{b,y},$  satisfies
%$$
%\mathcal{D}_{L_{\alpha_1}^b,L_{\alpha_2}^{b,y}}(H_1, H_2)=
%  \begin{cases}
 %   \frac{H_1}{{\alpha_1}},  & \quad {\rm if} \quad H_2=\frac{{\alpha_2}}{\alpha_1} H_1, \quad {\rm and} \quad \ H_1\in [0,  \alpha_1 ],\vspace{1ex}\\
 %   \min\left\{\frac{H_1}{\alpha_1}, \frac{H_2}{{\alpha_2}}\right\},       & \quad {\rm if}\ (H_1,H_2) \in \left[0,\frac{{\alpha_1}}{\Delta ^b(y)}\right]\times\left[0,\frac{{\alpha_2}}{\Delta ^b(y)}\right],\vspace{1ex}\\
%-\infty       & \quad {\rm if}\ H_1 >{\alpha_1} \quad {\rm or}\quad H_2 >{\alpha_2}.
 % \end{cases}
%$$
$$
\mathcal{D}_{L_{\alpha_1}^b,L_{\alpha_2}^{b,y}}(H_1, H_2)=
  \begin{cases}
    \min\left\{\frac{H_1}{\alpha_1}, \frac{H_2}{{\alpha_2}}\right\},       & \quad {\rm if}\ (H_1,H_2) \in K_{ \alpha_1,\alpha_2 }(y),\vspace{1ex}\\
-\infty       & \quad {\rm if} \ (H_1,H_2) \notin  [0, {\alpha_1}] \times  [0, {\alpha_2}].
  \end{cases}
$$
\end{theorem}

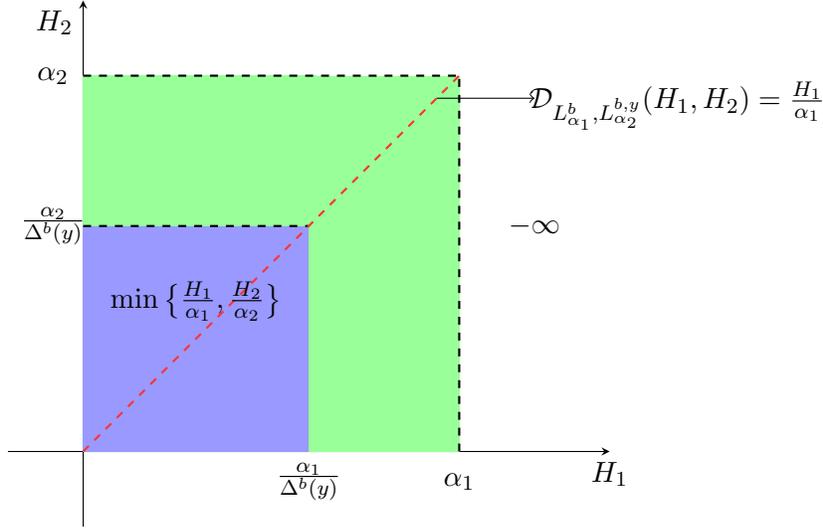
\begin{figure}[H]
\centering
\begin{tikzpicture}
\draw [-stealth](0,0) -- (8,0)node[below]{$H_1$};
\draw [-stealth](1,-1) -- (1,6)node[below,xshift=-1em]{$H_2$};
\node [black] at(4,-0.4){$\frac{{\alpha_1}}{\Delta ^b(y)}$};
\node [black] at(6,-0.4){${\alpha_1}$};
\node [black] at(0.6,3){$\frac{{\alpha_2}}{\Delta ^b(y)}$};
\node [black] at(0.6,5){${\alpha_2}$};
\draw[thick,dashed] (4,0) -- (4,3);
\fill[blue!40] (1,0) rectangle (4,3);
\fill[green!40] (4,0) rectangle (6,5);
\fill[green!40] (1,3) rectangle (4.1,5);
\draw[thick,dashed] (6,0) -- (6,5);
\draw[thick,dashed] (1,3) -- (4,3);
\draw[thick,dashed] (1,5) -- (6,5);
\draw[thick,dashed,red!80] (1,0) -- (6,5);
% \node [black] at(3.2,4.7){the dimension is not clear};
% \node [black,rotate=-90] at(5.8,2.4){the dimension is not clear};
\node [black] at(2.5,2){$\min\big\{\frac{H_1}{{\alpha_1}}, \frac{H_2}{{\alpha_2}}\big\}$};
%\node [black] at(7,2){$-\infty$};
\node [black] at(7,3){$-\infty$};
\draw[->] (5.7,4.7) -- (7,4.7);
\node [black] at(8.9,4.6){$\mathcal{D}_{L_{{\alpha_1}}^b,L_{\alpha_2}^{b,y}}(H_1, H_2)=\frac{H_1}{{\alpha_1}}$};
\end{tikzpicture}
\caption{The spectrum on the red line is $\frac{H_1}{{\alpha_1}}$, and the spectrum on the blue rectangle is $\min\left\{\frac{H_1}{{\alpha_1}}, \frac{H_2}{{\alpha_2}}\right\}.$} \label{fig:M9}
\end{figure}

The bivariate multifractal spectrum for almost all translations follows.

\begin{corollary}
For almost every $y$,  the bivariate multifractal spectrum  of the L\'evy function $L_{\alpha_1}^b$ and the translated  function $L_{\alpha_2}^{b,y}$ is
$$
\mathcal{D}_{L_{{\alpha_1}}^b,L_{{\alpha_2}}^{b,y}}(H_1, H_2)=
  \begin{cases}
    \min\left\{\frac{H_1}{{\alpha_1}}, \frac{H_2}{{\alpha_2}}\right\},       & \quad {\rm if}\ (H_1 , H_2)\in  [0,{\alpha_1}]\times \big[0,{\alpha_2}],\\
-\infty       & \quad {\rm else}.
  \end{cases}
$$
\end{corollary}

\begin{proof}
Let  \[E_{\alpha}=\left\{y:\ \left|y-\frac{k}{b^n}\right| \leqslant \frac{1}{b^{\alpha j}}\ \text{for\ infinite\ many  } n\right\},\] and
 \[A_{\eta}=\{y: \Delta ^b(y)=\eta\},\ (\eta\geqslant 1).\]
 Then  
\begin{equation}
\nonumber
A_{\eta}=\bigcap\limits_{\alpha\leqslant\eta} E_\alpha - \bigcup \limits_{\alpha>\eta} E_\alpha. 
\end{equation}
By a straightforward covering argument, see e.g. \cite[Theorem 2]{philipp1967some},  $E_1=\mathbb{R}$ and $m(E_\alpha)=0$ for any $\alpha>1$, where $m(\cdot)$ denote the Lebesgue measure. Hence, for almost all $y$, $\Delta ^b(y)=1$, and 
\[
K_{ \alpha_1,\alpha_2 }(y)=[0,{\alpha_1}]\times \big[0,{\alpha_2}].
\]
\end{proof}

%This is due to the fact that almost every $y$ belongs to $A_1$ ($m(A_1)=\infty $ which we now verify. 

% This is due to $ \mathcal{L}eb(\complement_{\mathbb{R}}A_{1})<\infty$, where $\complement_{\mathbb{R}}A_{1}$
%   represents the complement of the set $A_{1}$ in $\mathcal{R}$. By Definition \ref{def2}, naturally,
% \begin{equation}
%     \begin{aligned}
%         \nonumber
%         A_{\eta} \subset \{y: \Delta ^b(y)\leqslant \eta\} = \ \bigcap\limits_{n=1}^{\infty}\bigcup\limits_{j=n}^{\infty}\bigcup\limits_{p=0}^{b^j-1}\left[\frac{p}{b^j}-\frac{1}{b^{j\eta}}, \frac{p}{b^j}+\frac{1}{b^{j\eta}} \right]. 
%     \end{aligned}
% \end{equation} Since
% $
%     \mathcal{L}eb(A_{\eta})\leqslant\sum\limits_{n=1}^{\infty}b^n\cdot \frac{1}{b^{n\eta}}<\infty,
% $
% for any $\eta>1$, where $\mathcal{L}eb$ is the Lebesgue measure. Then 
% $\mathcal{L}eb\left(\bigcup\limits_{\eta>1}A_\eta\right)<\infty$. 

\begin{remark}
The bivariate multifractal spectrum $\mathcal{D}_{L_{\alpha_1}^b,L_{\alpha_2}^{b,y}}$ in the region $K_{ \alpha_1,\alpha_2 }(y)$ only depends on the $b$-adic approximation properties and is given by Theorem \ref{th} for any parameter $y\in \mathbb{R}$. But the spectrum in the region 
 $ [0,{\alpha_1}]\times [0,{\alpha_2}] \setminus K_{ \alpha_1,\alpha_2 }(y)$ (the green part in Figure \ref{fig:M9}), can be different for different values of $y$ which share the same $b$-adic approximation exponent $\Delta ^b(y)$: it actually  depends on the explicit expression of the $b$-ary expansion (see Subsection \ref{subsection2.2}) of $y$, as shown by Propositions \ref{prop3} and \ref{prop1}. The full determination of the  support of the bivariate multifractal spectra for all values of $y$ remains  to be settled. However, the following result shows that this spectrum can take two values only.  % unclear which may differ for different $y$.
\end{remark}

% { \color{blue} la remarque suivante est importante et meriterait un theoreme. Dire ou est sa demonstration. En sait-on plus?  Donner en remarque  les cas particuliers de $y$ inetressant ($y$ rationnel, presque tout $y$ , ...)  }

% For any $y\in \mathbb{R}$, we have the following result of the spectrum in the region $[0,{\alpha_1}]\times [0,{\alpha_2}] \setminus K_{ \alpha_1,\alpha_2 }(y)$. 

\begin{theorem}\label{greenpartspectrum}
    For any $(H_1,H_2)\in [0,{\alpha_1}]\times [0,{\alpha_2}]\mathbin{\Big\backslash} K_{ \alpha_1,\alpha_2 }(y)$, the bivariate multifractal spectrum  of the L\'evy function $L_{\alpha_1}^b$ and the translated  function $L_{\alpha_2}^{b,y}$ is
     either
    \begin{equation}
        \nonumber
        \mathcal{D}_{L_{{\alpha_1}}^b,L_{\alpha_2}^{b,y}}(H_1, H_2)=\min\left\{\frac{H_1}{{\alpha_1}}, \frac{H_2}{{\alpha_2}}\right\},
    \end{equation}
    or 
    \begin{equation}
        \nonumber
        \mathcal{D}_{L_{{\alpha_1}}^b,L_{\alpha_2}^{b,y}}(H_1, H_2)=-\infty.
    \end{equation}
\end{theorem}

The proof of Theorem \ref{greenpartspectrum} will be given in Section \ref{section8}.

\begin{remark}\label{examplebadicrational}
For any   rational number $y$, if $y$ is a $b$-adic rational number, then the bivariate multifractal spectrum of the L\'evy functions $L_{\alpha_1}^b$ and  $L_{\alpha_2}^{b,y}$ takes  the following values on the diagonal  $\frac{H_1}{{\alpha_1}} = \frac{H_2}{{\alpha_2}}$, 
\begin{equation}
    \begin{aligned}
        \nonumber
        \mathcal{D}_{L_{{\alpha_1}}^b,L_{\alpha_2}^{b,y}}(H_1, H_2)=
  \begin{cases}
    \frac{H_1}{{\alpha_1}},       & \quad {\rm if}\ H_2 = \frac{{\alpha_2}}{{\alpha_1}}H_1,\ \text{and}\ H_1\in [0,\alpha_1],\\
-\infty       & \quad {\rm else}.
  \end{cases}
    \end{aligned}
\end{equation}
See Figure \ref{binaryexam}.

If $y$ is a non-$b$-adic rational number, then the bivariate multifractal spectrum of the L\'evy functions $L_{\alpha_1}^b$ and  $L_{\alpha_2}^{b,y}$ is 
$$
\mathcal{D}_{L_{{\alpha_1}}^b,L_{{\alpha_2}}^{b,y}}(H_1, H_2)=
  \begin{cases}
    \min\left\{\frac{H_1}{{\alpha_1}}, \frac{H_2}{{\alpha_2}}\right\},       & \quad {\rm if}\ (H_1 , H_2)\in  [0,{\alpha_1}]\times \big[0,{\alpha_2}],\\
-\infty       & \quad {\rm else}.
  \end{cases}
$$

In fact, on one hand, if $y$ is $b$-adic rational, then the two L\'evy functions differ only by a finite sum of terms $\frac{\{b^iy\}}{b^{\alpha i}}$ and the result is immediate. On the other hand, if $y$ is non-$b$-adic rational, then $\Delta^b(y)=1$. Consequently, the result follows directly as a corollary of Theorem \ref{th}.  
\end{remark}
% Proposition \ref{examplebadicrational} will be proven in Section \ref{section6}.
\begin{figure}[H]
\centering
\begin{tikzpicture}[scale=0.75]
\draw [-stealth](0,0) -- (8,0)node[below]{$H_1$};
\draw [-stealth](1,-1) -- (1,6)node[below,xshift=-1em]{$H_2$};
\draw[thick,dashed] (6,0) -- (6,5);
\draw[thick,dashed] (1,5) -- (6,5);
\node [black] at(6,-0.4){${\alpha_1}$};
\node [black] at(0.6,5){${\alpha_2}$};
\draw[thick,dashed,red!80] (1,0) -- (6,5);
\draw[->] (5.5,4.5) -- (7,4.5);
\node [black] at(9.7,4.5){$\mathcal{D}_{L_{{\alpha_1}}^b,L_{\alpha_2}^{b,y}}(H_1, H_2)=\frac{H_1}{{\alpha_1}}$};
\node [black] at(5.3,2.5){$-\infty$};
\node [black] at(2.7,4.5){$-\infty$};
\node [black] at(7,3){$-\infty$};
\end{tikzpicture}
\caption{} \label{binaryexam}
\end{figure} 

% { \color{blue} Cette proposition est juste une remarque qu'il faut demontrer ici en quelques lignes: si $y$ est $b$-adique, les  deux fonctiosn de Levy ne different que d'une somme finie de termes $\frac{\{b^ix\}}{b^{\alpha i}}$ et le resultat est immediat. Et si $y$ n'est pas  $b$-adique, son exposant d'apprximation est 1, et le resultat est un corollaire du theoreme \ref{th}.}

% Et la remarque ci-dessous est a enlever: elle n'ajoute vraiment rien! } 

% \begin{remark}
% It gives "for free" that if $b=2$, and $y=\frac{1}{2}$, then the bivariate spectrum is as shown in Figure \ref{binaryexam}.  If $b=2,$ and $y=\frac{1}{3}$, then the bivariate spectrum is supported on  $[0,{\alpha_1}]\times [0, {\alpha_2}]$ and as shown in Figure \ref{fig:M5}. 
% \end{remark}

To end this section, we state the following  propositions which show that the bivariate multifractal spectrum depends on the explicit expression of the $b$-ary expansion of $y$ when $(H_1,H_2)$ satisfies $\min\left\{\frac{{\alpha_1}}{H_1} ,\frac{{\alpha_2}}{H_2}\right\}<\Delta ^b(y)$ (the green part in Figure \ref{fig:M9}). %Even if $\Delta ^b(y)$ is the same, the spectra may still be different. %As shown in the following two propositions. 

% { \color{blue} les examples ont une demonstration complexe: ils meritent un autre nom!}

\begin{proposition} \label{prop3}
Let $\eta> 1$.  Let $\{l_i\}$ be an increasing sequence of integers defined by 
    \begin{equation}\nonumber
        l_{i+1}=\lfloor l_i\eta \rfloor,\ \forall i\geqslant 1.
    \end{equation} 
    and $\nu_k=0^{\lfloor l_i\eta\rfloor-l_i-1}1$, where $l_1$ is a sufficiently large integer.
 If the $b$-ary expansion of $y$ is $\nu_1\nu_2\nu_3\cdots$ (see Figure \ref{prop3f}),
then  $\Delta ^b(y)=\eta$, and the bivariate multifractal spectrum of the L\'evy functions $L_{\alpha_1}^b$ and  $L_{\alpha_2}^{b,y}$ is
$$
\mathcal{D}_{L_{{\alpha_1}}^b,L_{\alpha_2}^{b,y}}(H_1, H_2)=
  \begin{cases}
    \frac{H_1}{{\alpha_1}},  & \ {\rm for} \quad \big(H_1, \frac{{\alpha_2}}{{\alpha_1}} H_1\big) \ {\rm with} \ H_1\in [0,  {\alpha_1} ],\vspace{1ex}\\
    \min\left\{\frac{H_1}{{\alpha_1}}, \frac{H_2}{{\alpha_2}}\right\},       & \ {\rm if}\ (H_1,H_2) \in \Big{[}0,\frac{{\alpha_1}}{\eta}\Big{]}\times \Big{[}0,\frac{{\alpha_2}}{\eta}\Big{]},\vspace{1ex}
    \\&\ {\rm or}\ H_1 \in \Big{[}\frac{\alpha_1(H_2)^2}{(\alpha_2)^2},\frac{{\alpha_1H_2}}{{\alpha_2}}\Big{]} \ {\rm with} \ H_2\in \left[\frac{\alpha_2}{\eta},  {\alpha_2} \right]
    ,\vspace{1ex}\\&\ {\rm or}\ H_2 \in \Big{[}\frac{\alpha_2(H_1)^2}{(\alpha_1)^2},\frac{{\alpha_2H_1}}{{\alpha_1}}\Big{]} \ {\rm with} \ H_1\in \left[\frac{\alpha_1}{\eta},  {\alpha_1} \right]\vspace{1ex}\\
-\infty       & \ {\rm else}.
  \end{cases}
$$
See Figure \ref{fig:M3}, for illustration.
\end{proposition}
\vspace{2mm}

 \begin{figure}[H]
\centering
\begin{tikzpicture}
\node[black] at (-2.2,0.3){$\cdots$};
\node[black] at(-1.9,0.3){$0$};
\node[black] at(-1.7,0.3){$0$};
\node[black] at(-1.3,0.3){$\cdots$};
\node[black] at(-0.9,0.3){$0$};
\node[black] at(-0.7,0.3){$0$};
\node[black] at(-0.5,0.3){$0$};
\node[red] at(-0.3,0.3){$1$};
\draw [decorate,decoration={brace,mirror}]
  (-1.9,0) -- (-0.4,0) node[midway,yshift=-1em]{{\footnotesize$\lfloor l_{i}\eta\rfloor -l_{i}-1$}};
\node[black] at(-0.1,0.3){$0$};
\node [black] at(0.1,0.3){$0$};
\node [black] at(0.3,0.3){$0$};
\node[black] at(0.7,0.3){$\cdots$};
\node [black] at(1.1,0.3){$0$};
\node [black] at(1.3,0.3){$0$};
\node [black] at(1.5,0.3){$0$};
\node [black] at(1.7,0.3){$0$};
\node [red] at(1.9,0.3){$1$};
\draw [decorate,decoration={brace,mirror}]
  (-0.1,0) -- (1.8,0) node[midway,yshift=-1em]{};
\node[black] at(2.1,0.3){$0$};
\node [black] at(2.3,0.3){$0$};
\node [black] at(2.5,0.3){$0$};
\node [black] at(2.7,0.3){$0$};
\node[black] at(3.1,0.3){$\cdots$};
\node [black] at(3.5,0.3){$0$};
\node [black] at(3.7,0.3){$0$};
\node [black] at(3.9,0.3){$0$};
\node [black] at(4.1,0.3){$0$};
\node [black] at(4.3,0.3){$0$};
\node [red] at(4.5,0.3){$1$};
\draw [decorate,decoration={brace,mirror}]
  (2.1,0) -- (4.4,0) node[midway,yshift=-1em]{};
  \node[black] at(4.7,0.3){$0$};
\node [black] at(4.9,0.3){$0$};
\node [black] at(5.1,0.3){$0$};
\node [black] at(5.3,0.3){$0$};
\node [black] at(5.5,0.3){$0$};
\node[black] at(5.9,0.3){$\cdots$};
\node [black] at(6.3,0.3){$0$};
\node [black] at(6.5,0.3){$0$};
\node [black] at(6.7,0.3){$0$};
\node [black] at(6.9,0.3){$0$};
\node [black] at(7.1,0.3){$0$};
\node [black] at(7.3,0.3){$0$};
\node [red] at(7.5,0.3){$1$};
\draw [decorate,decoration={brace,mirror}]
  (4.7,0) -- (7.4,0) node[midway,yshift=-1em]{{\footnotesize$\lfloor l_{i+3}\eta\rfloor -l_{i+3}-1$}};
  \node[black] at(7.9,0.3){$\cdots$};
\end{tikzpicture}
\caption{The $b$-ary expansion of $y$.}\label{prop3f}
\end{figure}

\begin{figure}[H]
\centering
\begin{tikzpicture}
\draw [-stealth](0,0) -- (8,0)node[below]{$H_1$};
\draw [-stealth](1,-1) -- (1,6)node[below,xshift=-1em]{$H_2$};
\node [black] at(4,-0.4){$\frac{{\alpha_1}}{\eta}$};
\node [black] at(6,-0.4){${\alpha_1}$};
\node [black] at(0.6,3){$\frac{{\alpha_2}}{\eta}$};
\node [black] at(0.6,5){${\alpha_2}$};
\draw[thick,dashed] (4,0) -- (4,1.5);
\fill[blue!40] (1,0) rectangle (4,3);
\fill[blue!40] (6,5)--(3.5,5)--(1,2.0)--(2.7,0.2)--(6,2) ;
\fill[yellow!40] (4,0)--(6,0)--(6,5)--(4,2.1);
\fill[yellow!40] (1,3)--(3.3,3)--(6,5)--(1,5);
\node [black] at(8.3,4.5){$H_2=\frac{\alpha_2H_1}{{\alpha_1}}$};
\draw[thick,dashed] (6,0) -- (6,5);
\draw[thick,dashed] (1,3) -- (4,3);
\draw[thick,dashed] (1,5) -- (6,5);
% \draw[thick,dashed] (1,0) -- (4,5);
% \draw[thick,dashed] (1,0)-- (6,3);
\draw[thick,dashed,red!80] (1,0) -- (6,5);
\node [black] at(2.2,4.7){$-\infty$};
\node [black] at(2.5,2){$\mathcal{D}=\min\big\{\frac{H_1}{{\alpha_1}}, \frac{H_2}{{\alpha_2}}\big\}$};
%  \draw[black] (1,0) arc (0:120:1 and 2);
% \path[green] (4, 5) edge [bend right] (1,0);
\draw (1,0) .. controls (1.5,1.6) .. (6,5);
\draw (1,0) .. controls (3,0.65) .. (6,5);
\node [black] at(10,3){$-\infty$};
\node [black] at(5,1.5){$-\infty$};
\draw[->] (4.1,2.3) -- (7,2.3);
\node [black] at(8.5,2.3){$H_2=\frac{\alpha_2(H_1)^2}{(\alpha_1)^2}$};
\draw[->] (3.6,3.2) -- (3.6,5.4);
\node [black] at(3.2,6){$H_1=\frac{\alpha_1(H_2)^2}{(\alpha_2)^2}$};
\draw[->] (5.5,4.5) -- (7,4.5);
\draw[thick,dashed] (4,0) -- (4,3);
% \draw[thick,dashed] (4,0.8) -- (6,0.8);
% \node [black] at(7.8,0.8){$H_2=\frac{{\alpha_2}}{\eta+\eta^2}$};
% \draw[thick,dashed] (1.5,3) -- (1.5,5);
% \node [black] at(2,6.8){$H_1=\frac{{\alpha_1}}{\eta+\eta^2}$};
\end{tikzpicture}
\caption{} \label{fig:M3}
\end{figure}

\begin{proposition}
    \label{prop1}
    Let $\{l_i\}$ be an increasing sequence of integers satisfying 
 \begin{equation}\nonumber
     \begin{aligned}
         \lim\limits_{i\to+\infty}\frac{l_{i-1}}{l_i}=0.
     \end{aligned}
 \end{equation}
 Let 
 \[
 \nu_i=
 \begin{cases}
   0^{\lfloor l_i\eta\rfloor-l_i}1(01)^{\lfloor(l_{i+1}-\lfloor l_i\eta\rfloor-1)/2\rfloor} & \text{if $(l_{i+1}-\lfloor l_i\eta\rfloor-1)/2$ is even,}\\
  0^{\lfloor l_i\eta\rfloor-l_i}1(01)^{\lfloor(l_{i+1}-\lfloor l_i\eta\rfloor-1)/2\rfloor}1  & \text{if $(l_{i+1}-\lfloor l_i\eta\rfloor-1)/2$ is odd}.
 \end{cases} 
 \]
 If the $b$-ary expansion of $y$ is $\nu_1\nu_2\nu_3\cdots$ (see Figure \ref{prop27y}),
then the bivariate multifractal spectrum of the L\'evy functions $L_{\alpha_1}^b$ and  $L_{\alpha_2}^{b,y}$ is
$$
\mathcal{D}_{L_{{\alpha_1}}^b,L_{\alpha_2}^{b,y}}(H_1, H_2)=
  \begin{cases}
    \min\left\{\frac{H_1}{{\alpha_1}}, \frac{H_2}{{\alpha_2}}\right\},       & \quad {\rm if}\ (H_1 , H_2)\in  [0,{\alpha_1}]\times \big[0,{\alpha_2}],\\
-\infty       & \quad {\rm else}.
  \end{cases}
$$
\end{proposition}

\begin{figure}[H]
\centering
\begin{tikzpicture}
\node[black] at(-3.6,0.3){$\cdots$};
\node[black] at(-3.3,0.3){$0$};
\node[black] at(-3.1,0.3){$0$};
\node[black] at(-2.3,0.3){$0$};
\node[black] at(-2.7,0.3){$\cdots$};
\node[red] at(-2.1,0.3){$1$};
\node[black] at(-1.9,0.3){$0$};
\draw [decorate,decoration={brace,mirror}]
  (-3.3,0) -- (-2.3,0) node[midway,yshift=-1em]{\footnotesize{$\lfloor l_i\eta\rfloor-l_i$}};
\node[black] at(-1.7,0.3){$1$};
\node[black] at(-1.5,0.3){$0$};
\node[black] at(-1.3,0.3){$1$};
\node[black] at(-1.1,0.3){$0$};
\node[black] at(-0.9,0.3){$1$};
\node[black] at(-0.7,0.3){$0$};
\node[black] at(-0.5,0.3){$1$};
\node[black] at(-0.3,0.3){$0$};
\node[black] at(-0.1,0.3){$1$};
\node [black] at(0.1,0.3){$0$};
\node [black] at(0.3,0.3){$1$};
\node [black] at(0.5,0.3){$0$};
\node[black] at(0.9,0.3){$\cdots$};
\node [black] at(1.3,0.3){$0$};
\node [black] at(1.5,0.3){$1$};
\node [black] at(1.7,0.3){$0$};
\node [black] at(1.9,0.3){$1$};
\draw [decorate,decoration={brace,mirror}]
  (-1.9,0) -- (1.9,0) node[midway,yshift=-1em]{\footnotesize{$l_{i+1}-\lfloor l_i\eta\rfloor-1$}};
\node[black] at(2.1,0.3){$0$};
\node [black] at(2.3,0.3){$0$};
\node [black] at(2.5,0.3){$0$};
\node[black] at(2.9,0.3){$\cdots$};
\node [black] at(3.3,0.3){$0$};
\node [black] at(3.5,0.3){$0$};
\node [red] at(3.7,0.3){$1$};
\draw [decorate,decoration={brace,mirror}]
  (2,0) -- (3.5,0) node[midway,yshift=-1em]{\footnotesize{$\lfloor l_{i+1}\eta\rfloor-l_{i+1}$}};
\node [black] at(3.9,0.3){$0$};
\node [black] at(4.1,0.3){$1$};
\node [black] at(4.3,0.3){$0$};
\node [black] at(4.5,0.3){$1$};
  \node[black] at(4.7,0.3){$0$};
\node [black] at(4.9,0.3){$1$};
\node [black] at(5.1,0.3){$0$};
\node [black] at(5.3,0.3){$1$};
\node [black] at(5.5,0.3){$0$};
\node [black] at(5.7,0.3){$1$};
\node [black] at(5.9,0.3){$0$};
\node [black] at(6.1,0.3){$1$};
\node [black] at(6.3,0.3){$0$};
\node[black] at(6.7,0.3){$\cdots$};
\node [black] at(7.1,0.3){$0$};
\node [black] at(7.3,0.3){$1$};
\node [black] at(7.5,0.3){$0$};
\node [black] at(7.7,0.3){$1$};
\node [black] at(7.9,0.3){$0$};
\node [black] at(8.1,0.3){$1$};
\node [black] at(8.3,0.3){$0$};
\node [black] at(8.5,0.3){$1$};
\draw [decorate,decoration={brace,mirror}]
  (3.8,0) -- (8.5,0) node[midway,yshift=-1em]{\footnotesize{$l_{i+2}-\lfloor l_{i+1}\eta\rfloor-1$}};
  \node[black] at(8.9,0.3){$\cdots$};
\end{tikzpicture}
\caption{The $b$-ary expansion of $y$ (here we take the example of $(l_{i+1}-\lfloor l_i\eta\rfloor-1)/2$ being an even number).}\label{prop27y}
\end{figure}
Propositions \ref{prop3} and \ref{prop1} will be proven in Section \ref{section7}.

\section{Preliminaries}\label{2}
%%%%%%%%%%%%%%%%%%%%%%%%%%%%%%%%%%%%%%%%%%%%%%%%%%%%%%%%%%%%%%%%%%%%%%%%%%%%%%%%%%%%%%%%%%%%%%%%%%%%%%%%%%%%%%%%%%%%%%%%%%%%%%%%%%%%%%%%%%%%%%%
\subsection{H\"older exponent of L\'evy functions and Diophantine approximation}\label{21}

Recall that the notions related to the $b$-adic approximation have been introduced in Definition \ref{def2}. 
It follows from \eqref{deltab}  that 
\begin{equation}
\Delta ^{b}(x-y)=\sup\left\{\alpha\geqslant1: \lVert b^n(x-y) \rVert \leqslant \frac{1}{b^{n(\alpha-1)}}{\rm\ for\ infinitely \ many}\ n\right\}.\label{dys2}
\end{equation}

The following proposition gives the H\"older exponent of the L\'evy function $L_{\alpha}^{b}$.

% { \color{blue} enlever ce qui concerne l'exposant de Holder et le spectre des translatees , qui ne doivent pas apparaitre comme des resultats}

\begin{proposition}\cite[Proposition 4]{jaffard1997old}\label{lemma1}
The H\"older exponent of $L_{\alpha}^{b}$ is
\begin{equation}\label{ls} \nonumber
h_{L_{\alpha}^{b}}(x)= \frac{{\alpha}}{\Delta ^{b}(x)}.
\end{equation}
\end{proposition}

On the other hand, observe that
\begin{equation} \label{2.5} 
	h_{L_{\alpha}^{b,y}}(x)=h_{L_{\alpha}^{b}}(x-y).
\end{equation}
So, by Proposition \ref{lemma1}, 
\begin{equation} \label{holderexponentoftranslatedfunction}
   h_{L_{\alpha}^{b,y}}(x) =\frac{{\alpha}}{\Delta ^{b}(x-y)}. 
\end{equation}

\begin{remark}
By Proposition \ref{lemma1} and \eqref{holderexponentoftranslatedfunction}, for any $[H_1,H_2]\in [0, {\alpha_1}]\times[0, {\alpha_2}]$, %the level set can be written as
\begin{equation}\label{eqqq} 
\begin{aligned}
E_{L_{{\alpha_1}}^b,L_{\alpha_2}^{b,y}}(H_1,H_2)&=\left\{x \in \mathbb{R}:\Delta ^{b}(x) = \frac{{\alpha_1}}{H_1}\right\}\cap\left\{x \in \mathbb{R}: \Delta^{b}(x-y)=\frac{{\alpha_2}}{H_2}\right\}\\
&=E_{L_{\alpha_1}^{b}}(H_1)\cap E_{L_{\alpha_2}^{b}}(H_2).
\end{aligned}
\end{equation}
\end{remark}

The main issue of the paper is to calculate the Hausdorff dimension of the level set $E_{L_{{\alpha_1}}^b,L_{\alpha_2}^{b,y}}(H_1,H_2)$ which is by \eqref{eqqq} an intersection of two fractal sets. However, for each of the two sets in the intersection, its  Hausdorff dimension is already known. 
The univariate multifractal spectrum $\mathcal{D}_{L_{{\alpha}}^b}$, i.e., the Hausdorff dimension of the first fractal set in the intersection \eqref{eqqq} was obtained by S. Jaffard \cite{jaffard1997old}.
\begin{proposition}\cite[Proposition 4]{jaffard1997old}\label{lemma3}
The multifractal spectrum $\mathcal{D}_{L_{{\alpha}}^b}:H\mapsto \dim_H\left(E_{L_{\alpha}^{b}}(H)\right)$ of the L\'evy function $L_{\alpha}^{b}$ is given by
\begin{equation}\label{spec1} \nonumber
\begin{aligned}
\mathcal{D}_{L_{{\alpha}}^b}(H)=
\left\{
             \begin{array}{ll}
            \frac{H}{{\alpha}},&\quad {\rm if}\ H \in [0,{\alpha}],\vspace{1ex}\\
            -\infty,&\quad{\rm else}.
             \end{array}
\right.
\end{aligned}
\end{equation}
\end{proposition}
Noting that by \eqref{2.5}, for any given $y\in \mathbb{R}$, we have  $E_{L_{\alpha}^{b,y}}=E_{L_{\alpha}^{b}}+ y $, it follows directly that the univariate multifractal spectrum $\mathcal{D}_{L_{{\alpha}}^{b,y}}$, i.e., the Hausdorff dimension of the second fractal set in the intersection \eqref{eqqq} is:
\begin{equation}\label{spec2} \nonumber
\begin{aligned}
\mathcal{D}_{L_{\alpha}^{b,y}}(H)=
\left\{
             \begin{array}{ll}
            \frac{H}{{\alpha}},&\quad {\rm if}\ H \in [0,{\alpha}],\vspace{1ex}\\
            -\infty&\quad{\rm else}.
             \end{array}
\right.
\end{aligned}
\end{equation}

\smallskip
The following Billingsley's lemma is useful for calculating Hausdorff dimension. Let $\mu$ be a finite Borel measure on $[0,1]$. For $x\in [0,1]$, let $I_n(x)$ be the $n$-th generation, half-open $b$-adic interval of the form $[{\frac{j-1}{b^n}}, \frac{j}{b^n})$ containing $x$. The lower local pointwise dimension $\underline{d}_{\mu}(x)$ at $x$ is defined by
\begin{equation} \label{lpd} 
\underline{d}_{\mu}(x)=\mathop{\liminf}\limits_{n\rightarrow \infty}\frac{\log\mu\big(I_n(x)\big)}{\log |I_n(x)|}.
\end{equation}

\begin{lemma}(Billingsley's lemma, \cite[p.17]{Bishop-Peres})\label{billings}
Let $E$ be a Borel subset of $[0,1]$ and $\mu$ a finite Borel measure on $[0,1]$.
\\(1) If $\underline{d}_{\mu}(x) \geqslant s$ for all $\mu$-almost all $x \in E$, and $\mu(E) >0$, then $\dim_H(E)\geqslant s$.
\\(2) If $\underline{d}_{\mu}(x) \leqslant s$ for all $x \in E$, then $\dim_H(E)\leqslant s$.
\end{lemma}
\subsection{\texorpdfstring{Symbolic space and $b$-}\ ary expansion }\label{subsection2.2}
Since L\'evy functions are  1-periodic, they can be considered as functions defined on the interval $[0,1)$. The level sets $E_{L_{{\alpha_1}}^b,L_{\alpha_2}^{b,y}}(H_1,H_2)$ will also be considered as their intersection with $[0,1)$.

Denote $\mathcal{A}:=\{0,1,\cdots,b-1\}$. Denote by $\mathcal{A}^n$ the set of all words of length $n$ over $\mathcal{A}$, and let $\mathcal{A}^*=\bigcup\limits_{n=1}^\infty\mathcal{A}^n$ be the set of all finite words over $\mathcal{A}$. 
 Every number $x$ in $[0,1)$ admits a $b$-ary expansion
\begin{equation} \nonumber
x=\sum\limits_{i=1}^{\infty} \frac{\varepsilon_i}{b^i},\quad \text{with} \  \varepsilon_i\ \in \mathcal{A}.
\end{equation}
Then a number $x \in [0,1)$ is associated with an infinite word $\varepsilon=\varepsilon_1\varepsilon_2\cdots$ which is also referred to as the $b$-ary expansion of $x$. In fact, there is a projection map $\Pi:\mathcal{A}^{\mathbb{N}}\longmapsto[0,1)$ defined by 
\begin{equation}\nonumber
\Pi(\varepsilon)=\sum\limits_{i=1}^{\infty} \frac{\varepsilon_i}{b^i},\quad \forall \varepsilon=\varepsilon_1\varepsilon_2\cdots\in\mathcal{A}^{\mathbb{N}}.
\end{equation}
The map $\Pi$ is one-to-one except for the $b$-adic rationals (where it is two-to-one). 

If $x$ is a $b$-adic rational number, denoted by $x=\frac{k}{b^m}$, $(\text{where}\ \gcd(k,b)=1)$, then the $b$-ary expansion $\varepsilon_1\varepsilon_2\cdots$ of $\frac{k}{b^m}$ satisfies one of the following two conditions, 
\begin{itemize}
\item $\varepsilon_m= a,\ \varepsilon_k=0,\ \forall k>m;$ 
\item
$\varepsilon_m = a-1,\ \varepsilon_k=b-1,\ \forall k>m, $
\end{itemize}
where $a\in \mathcal{A}\backslash\{0\}$.
In this paper, the $b$-ary expansion of the $b$-adic rational number $\frac{k}{b^m}$ is the one which satisfies $\varepsilon_m=a$, $\varepsilon_k=0$ for all $k>m$. 

\smallskip
To study the approximation properties in \eqref{deltab} and \eqref{dys2}, it is more convenient to work on the space $\Sigma:=\mathcal{A}^{\mathbb{N}}$ of infinite words. The space $\Sigma$ is naturally equipped with a metric defined by
\begin{equation} \nonumber
\forall \varepsilon, \varepsilon' \in \Sigma, \ d(\varepsilon,\varepsilon')=b^{-\min\{n\geqslant 0:\ \varepsilon_{n+1}\neq\varepsilon'_{n+1}\}}.
\end{equation}
Define the left shift $\sigma$ on $\Sigma$ by
\begin{equation} \nonumber
\sigma: \varepsilon_1\varepsilon_2\cdots\longmapsto\varepsilon_2\varepsilon_3\cdots.
\end{equation}
If the infinite word associated with $x$ is $\varepsilon$, then for any $n$, the infinite word associated with $b^nx$ is $\sigma^n(\varepsilon)$.  For any $a \in \mathcal{A}$, $k \in \mathbb{N}$, denote $a^k=\underbrace{aa\cdots a}_{k}$ and $a^\infty=aa\cdots$. Then, 
\begin{equation} \label{shun0} \nonumber
    \frac{1}{b}d\big(\sigma^{n}(\varepsilon),\varepsilon^*\big) \leqslant \|b^nx\|\leqslant d\big(\sigma^{n}(\varepsilon),\widetilde{\varepsilon}\big),
\end{equation}
where $\varepsilon^*,\ \widetilde{\varepsilon}=0^\infty $ or $(b-1)^\infty$. As a consequence, if $\sigma^n(\varepsilon)$ is followed by a word $\varepsilon_{n+1}\varepsilon_{n+2}\cdots$, where
\begin{equation}
\begin{aligned}
    \nonumber	
    &\varepsilon_{n+1}\varepsilon_{n+2}\cdots\varepsilon_{n+m}=0^{m},\\
    &\varepsilon_{n+m+1}\neq 0,
\end{aligned}
\end{equation}
or
\begin{equation}
\begin{aligned}
    \nonumber	
    &\varepsilon_{n+1}\varepsilon_{n+2}\cdots\varepsilon_{n+m}=(b-1)^{m},\\
    &\varepsilon_{n+m+1}\neq b-1,
\end{aligned}
\end{equation}
then
\begin{equation} \label{2.13}
	\frac{1}{b^{m+1}}\leqslant	\lVert b^{n}x\rVert \leqslant \frac{1}{b^{m}}.
\end{equation}
 Let $\theta=\theta_1\theta_2\cdots$ be the $b$-ary expansion of $y$. If 
\begin{equation}
\begin{aligned}
    \nonumber	
    &\varepsilon_{n+1}\varepsilon_{n+2}\cdots\varepsilon_{n+m}=\theta_{n+1}\theta_{n+2}\cdots\theta_{n+m},\\
    &\varepsilon_{n+m+1}\neq \theta_{n+m+1},
\end{aligned}
\end{equation}
then
\begin{equation} \label{2.13y}
	\frac{1}{b^{m+1}}=\frac{1}{b}d\big(\sigma^{n}(\varepsilon),\sigma^{n}(\theta)\big) \leqslant	\lVert b^{n}(x-y)\rVert \leqslant d\big(\sigma^{n}(\varepsilon),\sigma^{n}(\theta)\big)=\frac{1}{b^{m}}.
\end{equation}

For $\varepsilon_1\varepsilon_2\cdots\varepsilon_n \in \mathcal{A}^n$, denote by $I_n(\varepsilon_1\varepsilon_2\cdots\varepsilon_n)$ the cylinder of order $n$ composed of real numbers in $[0,1)$ whose $b$-ary expansions start with $\varepsilon_1\varepsilon_2\cdots\varepsilon_n$. Then, $I_n(\varepsilon_1\varepsilon_2\cdots\varepsilon_n)$ is an interval of length $b^{-n}$.
%%%%%%%%%%%%%%%%%%%%%%%%%%%%%%%%%%%%%%%%%%%%%%%%%%%%%%%%%%%%%%%%%%%%%%%%%%%%%%%%%%%%%%%%%%%%%%%%%%
\section{Multivariate multifractal spectra of L\'evy functions}\label{3}
%%%%%%%%%%%%%%%%%%%%%%%%%%%%%%%%%%%%%%%%%%%%%%%%%%%%%%%%%%%%%%%%%%%%%%%%%%%%%%%%%%%%%%%%%%%%%%%%%%
In this section, we prove Theorem \ref{th} which gives the multivariate multifractal spectra of a L\'evy function and its translations.

The upper bound of the spectrum in Theorem \ref{th} follows  from straightforward observations. The main effort in the proof lies in establishing the lower bound. The strategy for proving the lower bound proceeds as follows. We begin by constructing a Cantor type subset of the level set $ E_{L_{{\alpha_1}}^b,L_{\alpha_2}^{b,y}}(H_1,H_2) $ for any $ (H_1, H_2) \in [0, {\alpha_1}] \times [0, {\alpha_2}] $ (Subsections \ref{Construct-Cantor} and \ref{subsets-E}). Then, we will define a measure on the Cantor type subset and estimate its lower local dimensions (Subsection \ref{mass-local-dim}). Finally, we obtain the lower bound of the Hausdorff dimension of the Cantor type subset using Billingsley's lemma (Subsection \ref{proof-main-th}). Such lower bound is nothing but the wanted lower bound of the spectrum. 

\subsection{Construction of Cantor type subsets}\label{Construct-Cantor}

In this subsection, given data $\alpha_1, \alpha_2\geqslant 0$, $(H_1, H_2)\in [0, {\alpha_1}] \times [0, {\alpha_2}]$ and $y\in [0,1]$, we construct a Cantor type subset of $ E_{L_{{\alpha_1}}^b,L_{\alpha_2}^{b,y}}(H_1,H_2)$.   

%\begin{lemma} \label{existsubset}
%For any $ (H_1, H_2) \in [0, {\alpha_1}] \times [0, {\alpha_2}] $,  there exists a Cantor type subset of the level set $ E_{L_{{\alpha_1}}^b,L_{\alpha_2}^{b,y}}(H_1,H_2) $ of L\'evy functions $L_{{\alpha_1}}^b$ and $L_{\alpha_2}^{b,y}$.
%\end{lemma}
%\begin{proof}
Suppose $y \in[0,1]$ has $b$-ary expansion  $\theta=\theta_1\theta_2\cdots$. Let $\{n_k\}_{k\geqslant0}$ be a sequence of integers satisfying
\begin{equation}\label{314}
	\begin{aligned}
		\frac{k\cdot n_{k-1}}{n_{k}}\rightarrow 0,\ (k\rightarrow \infty).
	\end{aligned}
\end{equation}
For any $k \geqslant 0$, let 
\begin{equation} \label{313}
  \begin{aligned}
  &m_{k} = \left\lfloor n_{k}\left(\frac{{\alpha_1}}{H_1}-1\right)\right\rfloor,\\
  &m'_{k} = \left\lfloor n_{k}\left(\frac{{\alpha_2}}{H_2}-1\right)\right\rfloor,
  \end{aligned}
\end{equation}
and 
\begin{equation}\nonumber
    m^*_{k}=\min\{m_{k},m'_{k}\}.
\end{equation}
If $n_0\leqslant m^*_{0}$, then define
\begin{equation}
    \begin{aligned}\nonumber
        \mathcal{W}_{0}=\left\{	\varepsilon_{1}\varepsilon_{2}\cdots\varepsilon_{n_{0}}\in\mathcal{A}^*:\ \varepsilon_{n_{0}-1}=0\ \text{and}\ \varepsilon_{n_{0}}=1 \right\}.
    \end{aligned}
\end{equation}
If $n_0 > m^*_{0}$, then let
\begin{equation} \label{r0}
    r_{0}=\left\lfloor \frac{n_{0}-2}{ m^*_{0}}\right\rfloor,
\end{equation}
and define
\begin{equation}\nonumber
\begin{aligned}
\mathcal{W}_{0}=\Big\{
		\varepsilon_{1}\varepsilon_{2}\cdots\varepsilon_{n_{0}}\in\mathcal{A}^*:\ &\varepsilon_j=0 \ \text{for}\ j= n_{0}-1,\ \text{and}\  j=tm^*_{0}-1,\ 1\leqslant t \leqslant r_{0};\\ &\varepsilon_j=1\ \text{for}\  j= n_{0}, \ \text{and}\  j=tm^*_{0}\ 1\leqslant t \leqslant r_{0}  \Big\}.
\end{aligned}
\end{equation}
 Note that $2$ is subtracted  in \eqref{r0} to avoid the possibility of doubly defining the values of $\varepsilon_{n_0-1}$ and $\varepsilon_{n_0}$ in $\mathcal{W}_{0}$. %Specifically, 
%\begin{equation}
 %   \begin{aligned}
  %      \nonumber
   %     \varepsilon_{n_{0}-1}= 0,\ \varepsilon_{n_{0}}\neq 0,
    %\end{aligned}
%\end{equation}
%these two are the values that have been already determined. 

By \eqref{314} and \eqref{313}, without loss of generality, assume that
\begin{equation}\label{315}
	\begin{aligned}
		n_{2k+1}-n_{2k}-m_{2k}-8 > 0,
	\end{aligned}
\end{equation}
\begin{equation}\label{316}
	\begin{aligned}
		n_{2k+2}-n_{2k+1}-m'_{2k+1}-9 > 0.
	\end{aligned}
\end{equation}
Then, let 
\begin{equation}\label{317}\nonumber
\ r_{2k}=\left\lfloor \frac{n_{2k+1}-n_{2k}-m_{2k}-8}{m^*_{2k}}\right\rfloor,
\end{equation}
\begin{equation} \label{318}\nonumber
    r_{2k+1}=\left\lfloor \frac{n_{2k+2}-n_{2k+1}-m'_{2k+1}-9}{m^*_{2k+1}}\right\rfloor.
\end{equation}
For any $k\geqslant 0$, define
\begin{equation}\label{w2k}\nonumber
\begin{aligned}
\mathcal{W}_{2k+1}=\Big\{
		\varepsilon_{1}\varepsilon_{2}\cdots\varepsilon_{n_{2k+1}-n_{2k}}\in\mathcal{A}^*:\ &\varepsilon_j=0 \ \text{for}\ 1\leqslant j \leqslant m_{2k},\ j= n_{2k+1}-n_{2k}-2,\\&\ \ \ \ \ \ \ \ \ \text{and}\  j=m_{2k}+tm^*_{2k},\ 1\leqslant t \leqslant r_{2k};\\ &\varepsilon_j = 1\ \text{for}\  j=m_{2k}+tm^*_{2k}+1,\  0\leqslant t \leqslant r_{2k},\\&\ \ \ \ \ \ \ \ \ \text{and}\ j= n_{2k+1}-n_{2k}-1;
		\\&\varepsilon_j\neq\theta_{n_{2k}+j}\ \text{for}\  j=m_{2k}+tm^*_{2k}+2,
	\ 0\leqslant t\leqslant r_{2k},\\&\ \ \ \ \ \ \ \ \ \ \quad\quad\text{and}\ j= n_{2k+1}-n_{2k}  \Big\},
\end{aligned}
\end{equation}
and
\begin{equation}\label{w2k+1}\nonumber
	\begin{aligned}
		\mathcal{W}_{2k+2}=\Big\{
				\varepsilon_{1}\varepsilon_{2}\cdots\varepsilon_{n_{2k+2}-n_{2k+1}}\in\mathcal{A}^*: \ &\varepsilon_j=\theta_{n_{2k+1}+j} \ \text{for}\ 1\leqslant j \leqslant m'_{2k+1};\\& \varepsilon_j\neq\theta_{n_{2k+1}+j} \ \text{for}\ j=m'_{2k+1}+tm^*_{2k+1},\ 0\leqslant t \leqslant r_{2k+1},\\&\ \ \ \ \ \ \ \ \ \ \ \ \quad \quad\text{and}\  j= n_{2k+2}-n_{2k+1}-2;
				\\&\varepsilon_j=0 \ \text{for}\  j=m'_{2k+1}+tm^*_{2k+1}+1,\ 0\leqslant t\leqslant r_{2k+1},\\&\ \ \ \ \ \ \ \ \ \ \text{and}\ j= n_{2k+2}-n_{2k+1}-1;
				\\&\varepsilon_j = 1\  \text{for}\  j=m'_{2k+1}+tm^*_{2k+1}+2,\ 0\leqslant t\leqslant r_{2k+1},\\&\ \ \ \ \ \ \ \ \ \ \text{and}\ j= n_{2k+2}-n_{2k+1} \Big\}.
		\end{aligned}
	\end{equation}
	
 For better illustration,  the notation $\varepsilon_i = \overline\theta_i$ is used whenever $\varepsilon_i \neq \theta_i$. Hence, the words in $\mathcal{W}_{2k+1}$ begin with $0^{m_{2k}}1\overline\theta_{n_{2k}+m_{2k}+2}$, contain a subword $01\overline\theta_{n_{2k}+m_{2k}+tm^*_{2k}+2}$ at the positions $m_{2k}+tm^*_{2k}$, and end with $01\overline\theta_{n_{2k+1}}$. Similarly, the words in $\mathcal{W}_{2k+2}$ begin with $\theta_{n_{2k+1}+1}\cdots\theta_{n_{2k+1}+m'_{2k+1}}\overline\theta_{n_{2k+1}+m'_{2k+1}+1}01$, contain a subword $\overline\theta_{n_{2k+1}+m'_{2k+1}+tm^*_{2k+1}}01$ at the positions $m_{2k+1}+tm^*_{2k+1}$, and end with $\overline\theta_{n_{2k+2}-2}01$. See Figures \ref{fig:M1} and \ref{fig:M2} for illustrations.
\begin{figure}[H]
\centering
\begin{scaletikzpicturetowidth}{\textwidth}
\begin{tikzpicture}[scale=\tikzscale]
\draw  [black](-2,0)--(13,0);
\draw  [black](-2,0)--(-2,0.3);
\draw  [black](13,0)--(13,0.3);
\filldraw [black] (1.5,0) circle (2pt);
\node (b) at(1.5,-0.5){$\varepsilon_{m_{2k}}$};
\filldraw [black] (5.5,0) circle (2pt);
\node (c) at(5.5,-0.5){$\varepsilon_{m_{2k}+m^*_{2k}}$};
\node[black] at(10,0.3){$\cdots$};
\filldraw [black] (13,0) circle (2pt);
\node (z) at(13,-0.5){$\varepsilon_{n_{2k+1}-n_{2k}}$};
\node[black] (d) at(-1.9,0.3){$0$};
\node[black]  (e) at(-1.7,0.3){$0$};
\node[black]  (f) at(-1.5,0.3){$0$};
\node[black]  (h) at(-1.3,0.3){$0$};
\node[black]  (i) at(-1.1,0.3){$0$};
\node[black] (g) at(-0.7,0.3){$\cdots$};
\node[black]  (1) at(-0.3,0.3){$0$};
\node[black]  (2) at(-0.1,0.3){$0$};
\node[black]  (3) at(0.1,0.3){$0$};
\node[black]  (4) at(0.3,0.3){$0$};
\node[black]  (4) at(0.5,0.3){$0$};
\node[black]  (4) at(0.7,0.3){$0$};
\node[black]  (4) at(0.9,0.3){$0$};
\node[black]  (4) at(1.1,0.3){$0$};
\node[black]  (4) at(1.3,0.3){$0$};
\node[red] (j) at(1.5,0.3){$1$};
\node[red] (k) at(2.6,0.3){$\overline\theta_{n_{2k}+m_{2k}+2}$};
\node [red](0) at(5.3,0.3){$0$};
\node[red] (p) at(5.5,0.3){$1$};
\node[red] (q) at(7,0.3){$\overline\theta_{n_{2k}+m_{2k}+m^*_{2k}+2}$};

\node[red] (0) at(12.6,0.3){$0$};
\node[red] (p) at(12.8,0.3){$1$};
\node[red] (q) at(13.5,0.3){$\overline\theta_{n_{2k}+1}$};
\end{tikzpicture}
\end{scaletikzpicturetowidth}
\caption{The construction of $\mathcal{W}_{2k+1}$.} \label{fig:M1}
\end{figure}
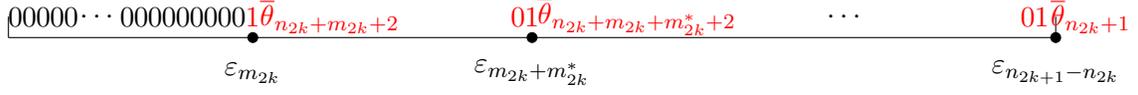
\begin{figure}[H]
\centering
\begin{scaletikzpicturetowidth}{\textwidth}
\begin{tikzpicture}[scale=\tikzscale]
\draw  [black](-2,0)--(13,0);
\draw  [black](-2,0)--(-2,0.3);
\draw  [black](13,0)--(13,0.3);
\filldraw [black] (2.3,0) circle (2pt);
\node (a) at(2.3,-0.5){$\varepsilon_{m_{2k+1}}$};
\filldraw [black] (6.6,0) circle (2pt);
\node (b) at(6.6,-0.5){$\varepsilon_{m_{2k+1}+m^*_{2k+1}}$};
\node[black] at(10.7,0.3){$\cdots$};
\filldraw [black] (13,0) circle (2pt);
\node (z) at(13,-0.5){$\varepsilon_{n_{2k+2}-n_{2k+1}}$};
\node[black] (d) at(-1.25,0.3){$\theta_{n_{2k+1}+1}$};
\node[black] (g) at(-0.3,0.3){$\cdots$};
\node[black]  (4) at(1,0.3){$\theta_{n_{2k+1}+m'_{2k+1}}$};
\node[red] (k) at(3.5,0.3){$\overline\theta_{n_{2k+1}+m'_{2k+1}+1}$};
\node[red] (l) at(4.9,0.3){$0$};
\node[red] (m) at(5.1,0.3){$1$};
\node[red] (k) at(8.2,0.3){$\overline\theta_{n_{2k+1}+m'_{2k+1}+m^*_{2k+1}}$};
\node[red] (l) at(10,0.3){$0$};
\node[red] (m) at(10.2,0.3){$1$};
\node[red] (k) at(12,0.3){$\overline\theta_{n_{2k+2}-2}$};
\node[red] (l) at(12.8,0.3){$0$};
\node[red] (m) at(13,0.3){$1$};
\end{tikzpicture}
\end{scaletikzpicturetowidth}
\caption{The construction of $\mathcal{W}_{2k+2}$. } \label{fig:M2}
\end{figure}
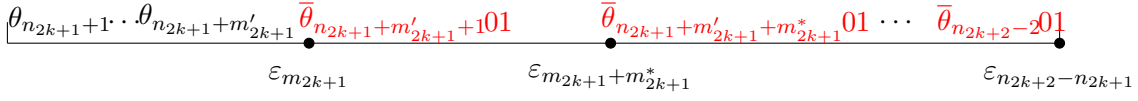
Notice that $8$ is subtracted in \eqref{315}, to well define $\mathcal{W}_{2k+1}$. Precisely, the following $8$ values should be defined only once: 
\begin{equation}
    \begin{aligned}
        \nonumber
        &\varepsilon_{m_{2k}+1}= 1,
        \\&\varepsilon_{m_{2k}+2}\neq\theta_{m_{2k}+2},
        \\&\varepsilon_{m_{2k}+r_{2k}m^*_{2k}}=0,
        \\&\varepsilon_{m_{2k}+r_{2k}m^*_{2k}+1}=1,
    \\&\varepsilon_{m_{2k}+r_{2k}m^*_{2k}+2}\neq\theta_{m_{2k}+r_{2k}m^*_{2k}+2},
    \\&\varepsilon_{n_{2k+1}-n_{2k}-2}=0,
    \\&\varepsilon_{n_{2k+1}-n_{2k}-1}=1,
    \\&\varepsilon_{n_{2k+1}-n_{2k}}\neq\theta_{n_{2k+1}-n_{2k}}.
    \end{aligned}
\end{equation}
Similarly, we subtract $9$ in \eqref{316}.

Take
\begin{equation}\nonumber
\begin{aligned}
\mathcal{U}_k=\{w_0w_1\cdots w_k:w_j\in \mathcal{W}_{j},\ 0\leqslant j\leqslant k\}.
\end{aligned}
\end{equation}
Our Cantor type subset is defined as
\begin{equation} \label{cantorsetE}
\begin{aligned}
E&:=\mathop{\bigcap}\limits_{k=0}^{\infty}
\mathop\bigcup_{
u \in \mathcal{U}_{k}}
I_{n_{k+1}}(u).
\end{aligned}
\end{equation}

\subsection{Subsets of \texorpdfstring{$E_{L_{{\alpha_1}}^b,L_{\alpha_2}^{b,y}}(H_1,H_2)$}{Lg}}\label{subsets-E}

In this subsection, we will prove that $E$ is a subset of $E_{L_{{\alpha_1}}^b,L_{\alpha_2}^{b,y}}(H_1,H_2)$. To this end, for any $x \in E$, we estimate $\lVert b^{n}x \rVert$ and $\lVert b^{n}(x-y) \rVert$ respectively. 

Note that for $x\in E$ with $\varepsilon=\varepsilon_{1}\varepsilon_{2}\varepsilon_{3}\cdots$ being its $b$-ary expansion, it holds that
\begin{equation} \nonumber
	\varepsilon=w_0w_1\cdots
\end{equation}
with $w_j\in \mathcal{W}_{j}$ for all $j\geqslant 0$.
\begin{itemize} 
\item The estimation for $\lVert b^{n}x \rVert$.
\begin{itemize} 
\item If $n=n_{2k}\ (k\geqslant 0),$ then by the definition of $\mathcal{W}_{2k+1}$,\ $\varepsilon_{n_{2k}}$ is followed by $0^{m_{2k}}1$. Hence, according to \eqref{2.13}, it follows that,
\begin{equation}\label{31} 
\frac{1}{b^{n_{2k}\big(\frac{{\alpha_1}}{H_1}-1\big)+1}} \leqslant \frac{1}{b^{m_{2k}+1}}\leqslant \lVert b^{n_{2k}}x \rVert \leqslant \frac{1}{b^{m_{2k}}}\leqslant \frac{1}{b^{n_{2k}\big(\frac{{\alpha_1}}{H_1}-1\big)-1}}.
\end{equation}
By \eqref{deltab}, it immediately follows that
\begin{equation} \label{3.7}
	\Delta^{b}(x)\geqslant \frac{{\alpha_1}}{H_1}.
\end{equation}
\item Let $n\in [n_{k},n_{k+1}]\ (k\geqslant 0)$. Note that $(H_1, H_2)\in K_{\alpha_1, \alpha_2}(y)$. Hence, $H_1\leqslant \frac{\alpha_1}{\Delta^b(y)}$. That is,
\[\Delta^b(y)\leqslant \frac{\alpha_1}{H_1}.\]
Thus, the prefix of the words in  $\mathcal{W}_{2k+2}$ can not contain a word of $0$'s or a word of $(b-1)$'s of length larger than $m_k$. Then, by the definitions of
$\mathcal{W}_{2k+1}$ and $\mathcal{W}_{2k+2}$, 
% combining the assumption that 
% \begin{equation} 
%     (H_1, H_2)\in \Big{[}0,\frac{{\alpha_1}}{\Delta ^b(y)}\Big{]}\times\Big{[}0,\frac{{\alpha_2}}{\Delta ^b(y)}\Big{]}, \label{assumption}
% \end{equation} 
$\varepsilon_n$ can not be followed by a word of $0$'s or a word of $(b-1)$'s of length larger than $m_k$. 
Moreover, by $\eqref{2.13}$,
\begin{equation}\label{34} 
	\begin{aligned}
		&\lVert b^{n}x \rVert \geqslant \frac{1}{b^{m_k}} \geqslant
		\frac{1}{b^{n_{k}\big(\frac{{\alpha_1}}{H_1}-1\big)}} > \frac{1}{b^{n\big(\frac{{\alpha_1}}{H_1}-1\big)}}.
	\end{aligned}
\end{equation}
Thus, 
\begin{equation}
\begin{aligned} \label{3.9}
\Delta^{b}(x) \leqslant \frac{{\alpha_1}}{H_1}.
\end{aligned}
\end{equation}
\end{itemize}
Combining \eqref{3.7} and \eqref {3.9} gives
\begin{equation} \label{lls1}
	\begin{aligned}
	\Delta^{b}(x)= \frac{{\alpha_1}}{H_1}.
	\end{aligned}
\end{equation}
\item The estimation for $\lVert b^{n}(x-y) \rVert$.
\begin{itemize}
    \item If $n=n_{2k+1}\ (k\geqslant 0)$, then by the definition of $\mathcal{W}_{2k+2}$, 
\begin{equation}\label{33} 
\frac{1}{b^{n_{2k+1}\big(\frac{{\alpha_2}}{H_2}-1\big)+1}} \leqslant \frac{1}{b^{m'_{2k+1}+1}}\leqslant \lVert b^{n_{2k+1}}(x-y) \rVert \leqslant \frac{1}{b^{m'_{2k+1}}}\leqslant \frac{1}{b^{n_{2k+1}\big(\frac{{\alpha_2}}{H_2}-1\big)-1}}.
\end{equation}
Based on \eqref{dys2}, it follows that
\begin{equation} \label{3.8}
	\Delta^{b}(x-y) \geqslant \frac{{\alpha_2}}{H_2}.
\end{equation}
\item Let $n\in [n_{k},n_{k+1}]\ (k\geqslant 0)$. Note that $(H_1, H_2)\in K_{\alpha_1, \alpha_2}(y)$. Hence, $H_2\leqslant \frac{\alpha_2}{\Delta^b(y)}$. i.e.,
\[\Delta^b(y)\leqslant \frac{\alpha_2}{H_2}.\]
Thus, the prefix of the words in  the $b$-ary expansion of $y$ can not contain a word of $0$'s of length larger than $m_k'$ in $\theta_{n_{2k}-n_{2k-1}}\theta_{n_{2k}-n_{2k-1}+1}\cdots\theta_{m_{2k}}$.  Then by the definitions of 
$\mathcal{W}_{2k}$ and $\mathcal{W}_{2k+1}$, $\varepsilon_n$ can not be followed by a word, as a prefix of $\theta_1\theta_2\cdots$, of length $m'_k$. By $\eqref{2.13y}$, it holds that
\begin{equation}\label{35} 
	\begin{aligned}
		&\lVert b^{n}(x-y) \rVert \geqslant \frac{1}{b^{m'_k}} \geqslant
		\frac{1}{b^{n_{k}\big(\frac{{\alpha_2}}{H_2}-1\big)}} > \frac{1}{b^{n\big(\frac{{\alpha_2}}{H_2}-1\big)}}.
	\end{aligned}
\end{equation}
Hence, based on \eqref{35}, it follows that
\begin{equation}
\begin{aligned} \label{3.10}
\Delta^{b}(x-y) \leqslant \frac{{\alpha_2}}{H_2}.
\end{aligned}
\end{equation}
\end{itemize}
Combining \eqref{3.8} and \eqref {3.10} yields
\begin{equation} \label{lls2}
	\begin{aligned}
		\Delta^{b}(x-y)= \frac{{\alpha_2}}{H_2}.
	\end{aligned}
\end{equation}
\end{itemize}
Therefore, by \eqref{lls1} and \eqref{lls2},
\begin{equation} \label{final}
E\subset E_{L_{{\alpha_1}}^b,L_{\alpha_2}^{b,y}}(H_1,H_2).
\end{equation}
% \end{proof}

\subsection{A mass distribution and estimation of the lower local pointwise  dimensions}\label{mass-local-dim}

In this subsection, we will first give a mass distribution $\mu$ on the subset $E$. Then, we estimate the lower local pointwise dimensions of the measure $\mu$. %Finally, we conclude Theorem \ref{th} by applying Billingsley's lemma.

% It remains to prove the lower bound of the spectrum. The Cantor type subset $E$ of $E_{L_{{\alpha_1}}^b,L_{\alpha_2}^{b,y}}(H_1,H_2)$
% , as given by \eqref{cantorsetE},
% is constructed in the proof of Lemma \ref{existsubset} whose Hausdorff dimension is more convenient to estimate. 

%Subsequently, the lower bound of the Hausdorff dimension of $E$ will be determined. 
\smallskip
Let us give a mass distribution $\mu$ on the Cantor type set $E$. To this end, by Kolmogorov's extension theorem (see for example, \cite[Theorem 2.4.4]{tao2011introduction}), it suffices to define the mass on each cylinder. Let $x \in E$, and let $I_n(x)$ be the cylinder of order $n$ containing $x$. Let $\varepsilon= \varepsilon_1\varepsilon_2\varepsilon_3\cdots$ be the $b$-ary expansion of $x$. Denote by $u_n(x)$ the number of "free positions" $i$ between $1$ and $n$, at which  $\varepsilon_i$ can be chosen in $\mathcal{A}$ freely, i.e., $\varepsilon_i$ is not fixed by the definition of $x \in E$. Then we define
\begin{equation} \label{3.16}
	\mu\big(I_{n}(x)\big)=b^{-u_{n}(x)}.
\end{equation}

In the following, the lower local pointwise dimension $\underline{d}_{\mu}(x)$ as defined in \eqref{lpd} of $\mu$ at a point $x\in E$ will be estimated.  First, by \eqref{3.16}, it holds that for any $x\in E$,
\begin{equation} \label{lld}\nonumber
	\begin{aligned}
	\underline{d}_{\mu}(x)=	\mathop{\liminf}\limits_{n\rightarrow \infty}\frac{\log\mu(I_{n}(x))}{\log \lvert I_{n}(x)\rvert}
	=\mathop{\liminf}\limits_{n\rightarrow \infty}\frac{\log b^{-u_n(x)}}{\log b
			^{-n}}
		=\mathop{\liminf}\limits_{n\rightarrow \infty}\frac{ u_{n}(x)}{{n}}.
	\end{aligned}
\end{equation}
To further estimate $u_{n}(x)$, we distinguish the following cases.
\begin{itemize}
\item If $n_{2k} \leqslant n\leqslant n_{2k}+m_{2k}+2$, then by the construction of $\mathcal{W}_{2k+1}$,
\begin{equation} \nonumber
	u_{n}(x)=u_{n_{2k}+m_{2k}}(x).
\end{equation}
Hence,
\begin{equation} \nonumber
\begin{aligned}
\frac{u_{n}(x)}{n}=\frac{u_{n_{2k}+m_{2k}}(x)}{n} \geqslant \frac{u_{n_{2k}+m_{2k}}(x)}{n_{2k}+m_{2k}+2}.
\end{aligned}
\end{equation}
\item If $n_{2k}+m_{2k}+3\leqslant n < n_{2k}+m_{2k}+m^*_{2k}$, then by the construction of $\mathcal{W}_{2k+1}$,
\begin{equation} \nonumber
	\begin{aligned}
	u_{n}(x)=u_{n_{2k}+m_{2k}}(x)+ n - n_{2k}- m_{2k}-2.\\
	\end{aligned}
\end{equation}
Note that $n-n_{2k}- m_{2k}-2>0$. Combing the fact that
\begin{equation} \label{fact}
\frac{x+a}{y+a}\geqslant\frac{x}{y},\quad \forall a > 0,\ x\leqslant y,
\end{equation}
it follow that
\begin{equation} \nonumber
	\begin{aligned}
		\frac{\log\mu(I_n(x))}{\log \lvert I_n(x)\rvert}=\frac{u_{n}(x)}{n} =\frac{u_{n_{2k}+m_{2k}}(x)+ n - n_{2k}- m_{2k}-2}{n_{2k}+m_{2k}+2+n-n_{2k}- m_{2k}-2}\geqslant\frac{u_{n_{2k}+m_{2k}}(x)}{n_{2k}+m_{2k}+2}.
	\end{aligned}
\end{equation}
\item If $n_{2k}+m_{2k}+m^*_{2k}\leqslant n < n_{2k+1}$, the same discussions as in the above two cases always apply.
Hence, for $n_{2k}\leqslant n< n_{2k+1},$ it can be concluded that
\begin{equation} \label{3.32}
		\frac{\log\mu(I_n(x))}{\log \lvert I_n(x)\rvert}	\geqslant\frac{u_{n_{2k}+m_{2k}}(x)}{n_{2k}+m_{2k}+2}.
\end{equation}
\end{itemize}

Note that
\begin{equation}\nonumber
	u_{n_{2k}+m_{2k}}(x)=n_{2k}-\sum\limits_{i=0 }^{k-1}\big(m_{2i}+5\big)-\sum\limits_{i=0 }^{k-1} \big(m'_{2i+1}+6\big)-3\sum\limits_{i=0 }^{2k-1}r_i.
\end{equation}
On the one hand,
\begin{equation} \nonumber
	\begin{aligned}
		\sum\limits_{i=0 }^{2k-1}r_i =&\sum\limits_{i=0 }^{k-1} \left\lfloor \frac{n_{2k+1}-n_{2k}-m_{2k}-8}{m^*_{2k}}\right\rfloor + \sum\limits_{i=0 }^{k-1} \left\lfloor \frac{n_{2k+2}-n_{2k+1}-m'_{2k+1}-9}{m^*_{2k+1}}\right\rfloor
		\\ \leqslant&
		\sum\limits_{i=0 }^{2k-1}\frac{n_{i+1}-n_{i}-m^*_{i}}{m^*_{i}}
		\\ \leqslant &
		\sum\limits_{i=0 }^{2k-1}\frac{n_{i+1}}{n_{i}\cdot\left(\min\left\{\frac{{\alpha_1}}{H_1}, \frac{{\alpha_2}}{H_2}\right\}-1\right)}
		\\ \leqslant & 2k\cdot\frac{n_{2k}}{n_{2k-1}\cdot\left(\min\left\{\frac{{\alpha_1}}{H_1}, \frac{{\alpha_2}}{H_2}\right\}-1\right)},
	\end{aligned}
\end{equation}
and by \eqref{314}, it holds that
\begin{equation} \nonumber
	\frac{\sum\limits_{i=0 }^{2k-1}3r_i}{n_{2k}} \rightarrow 0,\quad (k\rightarrow \infty).
\end{equation}
On the other hand,
\begin{equation} \nonumber
	\begin{aligned}
		\sum\limits_{i=0 }^{k-1}\left(m_{2i}+5\right)+\sum\limits_{i=0 }^{k-1} \left(m'_{2i+1}+6\right)&=\sum\limits_{i=0 }^{k-1}\left(\left\lfloor n_{2i}\left(\frac{{\alpha_1}}{H_1}-1\right)\right\rfloor+5\right)+\sum\limits_{i=0 }^{k-1}\left (\left\lfloor n_{2i+1}\left(\frac{{\alpha_2}}{H_2}-1\right)\right\rfloor+6\right)
		\\& \leqslant 2k\cdot n_{2k-1}\cdot\left(\max\left\{\frac{{\alpha_1}}{H_1}, \frac{{\alpha_2}}{H_2}\right\}-1\right) + 11k,
	\end{aligned}
\end{equation}
and by \eqref{314}, it follows that
\begin{equation} \nonumber
	\frac{\sum\limits_{i=0 }^{k-1}\big(m_{2i}+5\big)+\sum\limits_{i=0 }^{k-1} \big(m'_{2i+1}+6\big)}{n_{2k}} \rightarrow 0,\quad (k\rightarrow \infty).
\end{equation}
Therefore,
\begin{equation} \label{444}
	\begin{aligned}
		\mathop{\lim}\limits_{k\rightarrow \infty}\frac{u_{n_{2k}+m_{2k}}(x)}{n_{2k}+m_{2k}+2}
		=&\mathop{\lim}\limits_{k\rightarrow \infty}\frac{n_{2k}-\sum\limits_{i=0 }^{k-1}\big(m_{2i}+5\big)-\sum\limits_{i=0 }^{k-1} \big(m'_{2i+1}+6\big)-3\sum\limits_{i=0 }^{2k-1}r_i}{n_{2k}+m_{2k}+2}
		\\=&\mathop{\lim}\limits_{k\rightarrow \infty}\frac{n_{2k}}{n_{2k}+m_{2k}+2}
			\\=&\mathop{\lim}\limits_{k\rightarrow \infty}\frac{n_{2k}}{n_{2k}+\left\lfloor n_{2k}\left(\frac{{\alpha_1}}{H_1}-1\right)\right\rfloor+2}
		\\= &\frac{H_1}{{\alpha_1}}.
	\end{aligned}
\end{equation}

By the same arguments as the case $n_{2k}\leqslant n< n_{2k+1},$ it can be concluded that for $n_{2k+1}\leqslant n< n_{2k+2},$
\begin{equation} \label{3.33}
	\frac{\log\mu(I_n(x))}{\log \lvert I_n(x)\rvert}	\geqslant\frac{u_{n_{2k+1}+m'_{2k+1}}(x)}{n_{2k+1}+m'_{2k+1}+3},
\end{equation}
and \begin{equation} \label{333}
	\begin{aligned}
		\mathop{\lim}\limits_{k\rightarrow \infty}\frac{u_{n_{2k+1}+m'_{2k+1}}(x)}{n_{2k+1}+m'_{2k+1}+3}
		=\mathop{\lim}\limits_{k\rightarrow \infty}\frac{n_{2k+1}}{n_{2k+1}+m'_{2k+1}+3}
		= \frac{H_2}{{\alpha_2}}.
	\end{aligned}
\end{equation}

Combining \eqref{3.32}, \eqref{444}, \eqref{3.33} and \eqref{333}, we have that for any $x\in E$,
\begin{equation} \label {3.34}
	\begin{aligned}
\underline{d}_{\mu}(x)&=	\mathop{\liminf}\limits_{n\rightarrow \infty}\frac{\log\mu(I_n(x))}{\log \lvert I_n(x)\rvert}\\
&\geqslant \min\left\{\mathop{\lim}\limits_{k\rightarrow \infty}\frac{u_{n_{2k}+m_{2k}}(x)}{n_{2k}+m_{2k}+2},  \ \mathop{\lim}\limits_{k\rightarrow \infty}\frac{u_{n_{2k+1}+m'_{2k+1}}(x)}{n_{2k+1}+m'_{2k+1}+3}\right\}\\&= \min \left\{\frac{H_1}{{\alpha_1}}, \frac{H_2}{{\alpha_2}} \right\}. 
	\end{aligned}
\end{equation}
In fact,  it has been proven that when $k$ is large enough, between $n_{2k}$ and $n_{2k+2}$, the sequence $\frac{\log\mu\left(I_n(x)\right)}{\log \lvert I_n(x)\rvert}$ has two local minima at positions $n_{2k}+m_{2k}+2$ and $n_{2k+1}+m'_{2k+1}+3$. However, the limits exist for both subsequences $\{n_{2k}+m_{2k}+2\}_{k\geqslant0}$ and $\{n_{2k+1}+m'_{2k+1}+3\}_{k\geqslant0}$. Hence, the lower limit of the sequence $\frac{\log\mu(I_n(x))}{\log \lvert I_n(x)\rvert}$ is the minimum of these two limits. The analysis will be illustrated in Figure \ref{fig:my_label}.

\begin{figure}[H]
\begin{overpic}[width=15cm,height=9.5cm]{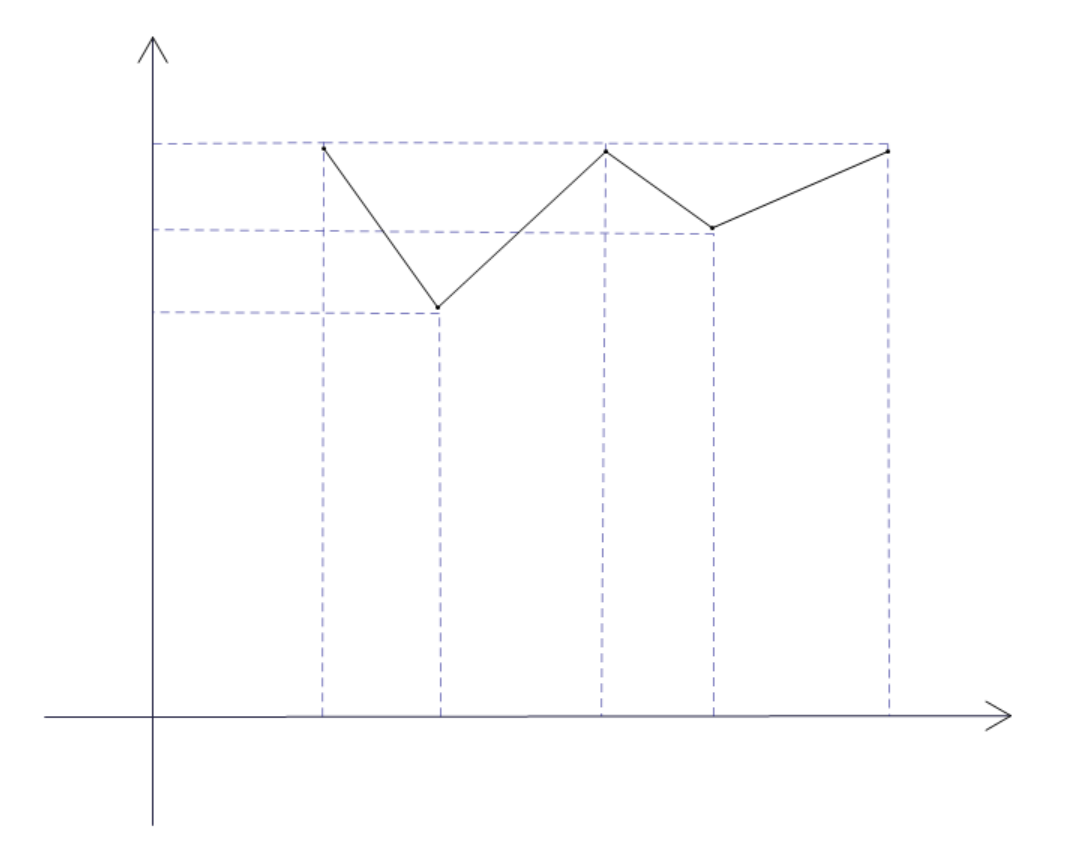}
	\put(2,58){\small \bfseries{$\frac{\log\mu\big(I_n(x)\big)}{\log \lvert I_{n}(x)\rvert}$}}
	\put(11,50){\small \bfseries{$1$}}
	\put(10,44){\small \bfseries{$\frac{H_2}{{\alpha_2}}$}}
	\put(10,39){\small \bfseries{$\frac{H_1}{{\alpha_1}}$}}
	\put(28,8){\small \bfseries{$n_{2k}$}}
	\put(37,8){\small \bfseries{$n_{2k}\!+\!m_{2k}\!+\!2$}}
	\put(53,8){\small \bfseries{$n_{2k+1}$}}
	\put(60,8){\small \bfseries{$n_{2k+1}\!+\!m'_{2k+1}\!+
 \!3$}}
	\put(79,8){\small \bfseries{$n_{2k+2}$}}
	\put(95,8){\small \bfseries{$n$}}
\end{overpic}
\vspace{-.7cm}
    \caption{The asymptotic behavior of the sequence $\left\{\frac{\log\mu(I_n(x))}{\log \lvert I_{n}(x)\rvert}\right\}_{n\geqslant 1}$.}
    \label{fig:my_label}
\end{figure}
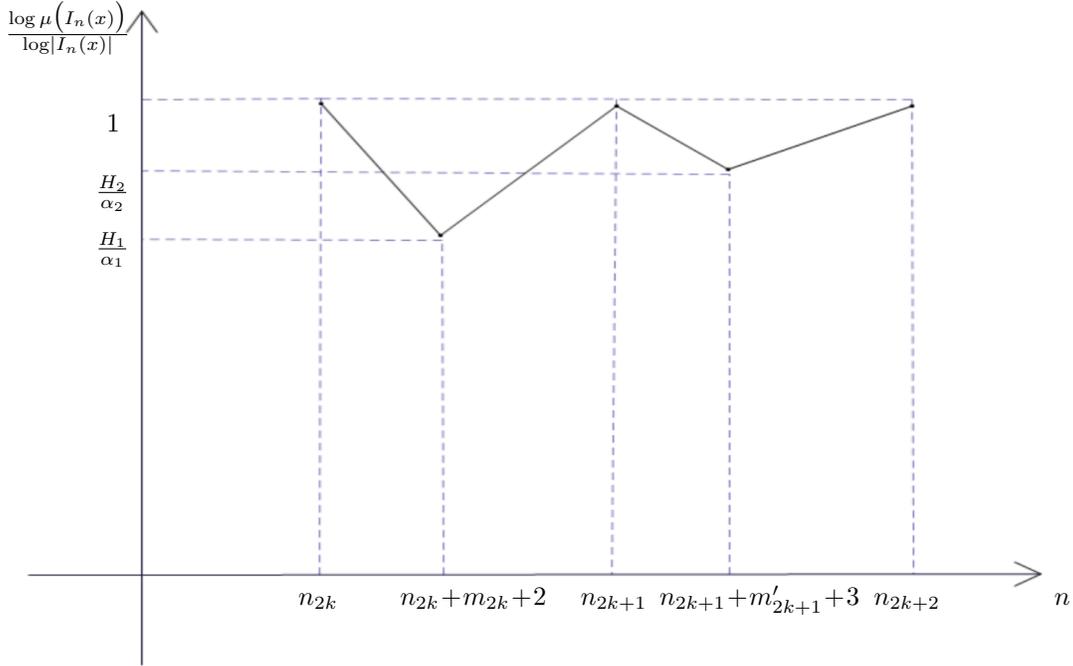

\subsection{Proof of Theorem \ref{th}}\label{proof-main-th}

Now, we are ready to give the final proof of Theorem \ref{th} by applying Billingsley's lemma.

\begin{proof}[Proof of Theorem \ref{th}]

First note that the second part of the equality is immediate since the support of a bivariate spectrum trivially is included in the product of the supports of each function. 

As regards the first part of the equality, the upper bound is easy to obtain; indeed note that
\begin{equation} \nonumber
\begin{aligned}
E_{L_{{\alpha_1}}^b,L_{\alpha_2}^{b,y}}(H_1,H_2)\subset E_{L_{{\alpha_1}}^b}(H_1),\\
E_{L_{{\alpha_1}}^b,L_{\alpha_2}^{b,y}}(H_1,H_2)\subset E_{L_{\alpha_2}^{b,y}}(H_2).
\end{aligned}
\end{equation}
Then, by Proposition \ref{lemma3}, for any $(H_1, H_2) \in [0,{\alpha_1}]\times [0,{\alpha_2}]$, the bivariate multifractal spectrum of L\'evy functions $L_{{\alpha_1}}^b$ and $L_{\alpha_2}^{b,y}$ satisfies
\begin{equation}\label{upp}
\dim_H\left(E_{L_{{\alpha_1}}^b,L_{\alpha_2}^{b,y}}(H_1,H_2)\right)\leqslant \min\left\{\frac{H_1}{{\alpha_1}},\frac{H_2}{{\alpha_2}}\right\}.
\end{equation}

For the lower bound, when $(H_1, H_2)\in \Big{[}0,\frac{{\alpha_1}}{\Delta ^b(y)}\Big{]}\times\Big{[}0,\frac{{\alpha_2}}{\Delta ^b(y)}\Big{]}$, it follows by combining \eqref{final}, \eqref{3.34} and Lemma \ref{billings} that %for $(H_1, H_2)\in \Big{[}0,\frac{{\alpha_1}}{\Delta ^b(y)}\Big{]}\times\Big{[}0,\frac{{\alpha_2}}{\Delta ^b(y)}\Big{]}$,
\begin{equation} \label{low}
\dim_H\left(E_{L_{{\alpha_1}}^b,L_{\alpha_2}^{b,y}}(H_1,H_2)\right)\geqslant\min\left\{\frac{H_1}{{\alpha_1}},\frac{H_2}{{\alpha_2}}\right\}.
\end{equation}
When $H_2=\frac{{\alpha_2}}{{\alpha_1}} H_1$, it follows that
\begin{equation}
    \nonumber
    E_{L_{{\alpha_1}}^b,L_{\alpha_2}^{b,y}}(H_1,H_2)\supset E_{L_{{\alpha_1}}^b}(H_1).
\end{equation}
Hence,
\begin{equation}\label{h2=alpha2/alpha1h1}
\dim_H\left(E_{L_{{\alpha_1}}^b,L_{\alpha_2}^{b,y}}(H_1,H_2)\right)\geqslant\dim_H\left(E_{L_{{\alpha_1}}^b}(H_1)\right)=\frac{H_1}{{\alpha_1}}.
\end{equation}

Finally, Theorem \ref{th} follows directly from \eqref{upp}, \eqref{low} and \eqref{h2=alpha2/alpha1h1}.
\end{proof}

 \section{Dichotomy of the size of \texorpdfstring{$E_{L_{{\alpha_1}}^b,L_{\alpha_2}^{b,y}}(H_1,H_2)$}{lg}}\label{section8}
 %%%%%%%%%%%%%%%%%%%%%%%%%%%%%%%%%%%%%%%%%%%%%%%%%%%%%%%%%%%%%%%%%%%%%%%%%%%%%%%%%%%%%%%%%%%%%%%%%%%%%%%%%%%%%%%%%%%%%%%%%%%%%%%%%%%%%%%%%%%%%%%
 In this section, we prove Theorem \ref{greenpartspectrum} which shows that the sets  $E_{L_{{\alpha_1}}^b,L_{\alpha_2}^{b,y}}(H_1,H_2)$ are either empty or of maximal dimension.
 \begin{proof}[Proof of Theorem \ref{greenpartspectrum}]
 Let us consider the level set 
\begin{equation} \nonumber
\begin{aligned}
E_{L_{{\alpha_1}}^b,L_{\alpha_2}^{b,y}}(H_1,H_2)&=\left\{x \in \mathbb{R}:\Delta^{b}(x) = \frac{{\alpha_1}}{H_1}, \Delta^{b}(x-y)=\frac{{\alpha_2}}{H_2}\right\}, 
\end{aligned}
\end{equation}
where $(H_1,H_2)\in[0,{\alpha_1}]\times [0,{\alpha_2}] \mathbin{\Big\backslash} K_{\alpha_1,\alpha_2}(y)$.

 The $b$-ary expansion of $y$ is written as $\theta_1\theta_2\theta_3\cdots$. By revisiting the proof of Theorem \ref{th}, the lower bound is based on the construction of a Cantor type set which is defined by using the sequence $\{n_k\}$. That is, once there exists a sequence which may depend on $y$ and $(H_1,H_2)$, such that $\frac{kn_k}{n_{k+1}}\to 0,(k\to\infty),$ and
\begin{enumerate}[label=(\roman*)]
    \item \label{c1}$\forall l\in \left[n_{2k},n_{2k}+n_{2k}\left\lfloor\left(\frac{\alpha_1}{H_1}-1\right)\right\rfloor\right]$, the length of the block of $0$'s followed by $\theta_l$ is less than or equal to $\left\lfloor l\left(\frac{\alpha_2}{H_2}-1\right)\right\rfloor$.
    \item \label{c2}$\forall l\in \left[n_{2k+1},n_{2k+1}+n_{2k+1}\left\lfloor\left(\frac{\alpha_2}{H_2}-1\right)\right\rfloor\right]$, the length of the block of $0$'s followed by $\theta_l$ is less than or equal to $\left\lfloor l\left(\frac{\alpha_1}{H_1}-1\right)\right\rfloor$.
\end{enumerate}
There will be two cases.

\textbf{Case 1}: ${H_2}<\frac{\alpha_2H_1}{\alpha_1}$. 

Suppose there exist infinitely many $n$ such that for all ${l\in \left[n, \left\lfloor n\left(\frac{\alpha_2}{H_2}-1\right)\right\rfloor+n\right]}$, the length of the block of 0's in the $b$-ary expansion of $y$ followed by $\theta_l$ is less than or equal to $\left\lfloor l\left(\frac{\alpha_1}{H_1}-1\right)\right\rfloor$. Then $n_{2k}$ can be taken as an infinite subsequence of these $n$'s. Hence, condition \ref{c2} holds. On the other hand, by $H_2<\frac{\alpha_2H_1}{\alpha_1}$, condition \ref{c1} holds automatically. Thus, by the same proof as in Theorem \ref{th}, A Cantor type subset of $E_{L_{{\alpha_1}}^b,L_{\alpha_2}^{b,y}}(H_1,H_2)$ can be constructed, and it follows that
$
\mathcal{D}_{L_{{\alpha_1}}^b,L_{{\alpha_2}}^{b,y}}(H_1, H_2)=
    \min\left\{\frac{H_1}{{\alpha_1}}, \frac{H_2}{{\alpha_2}}\right\}.
$

 Otherwise, for any sufficiently large $n$, there exist $l_n \in \left[n, \left\lfloor n\left(\frac{\alpha_2}{H_2}-1\right)\right\rfloor+n\right]$, such that the length of the block of 0's followed by $\theta_{l_n}$ is larger than $\left\lfloor l_n\left(\frac{\alpha_1}{H_1}-1\right)\right\rfloor$. For any $x\in E_{L_{{\alpha_1}}^b,L_{\alpha_2}^{b,y}}(H_1,H_2)$, it holds that $\Delta^b(x-y)=\frac{\alpha_2}{H_2}$. Then, there exist infinitely many $n'$ such that the $b$-ary expansion of $x$ satisfies $$\varepsilon_{n'+1}\varepsilon_{n'+2}\cdots\varepsilon_{n'+\left\lfloor n'\left(\frac{\alpha_2}{H_2}-1\right)\right\rfloor}=\theta_{n'+1}\theta_{n'+2}\cdots\theta_{n'+\left\lfloor n'\left(\frac{\alpha_2}{H_2}-1\right)\right\rfloor}.$$ 
Thus, for infinitely many $n'$,
\begin{equation} \nonumber
   \lVert b^{n'}x \rVert <  \frac{1}{b^{ \left\lfloor n'\left(\frac{{\alpha_1}}{H_1}-1\right)\right\rfloor}},
\end{equation}
which implies that 
\begin{equation}
    \nonumber
    \Delta^b(x)> \frac{\alpha_1}{H_1}.
\end{equation}
This contradicts the definition of the level set $E_{L_{{\alpha_1}}^b,L_{\alpha_2}^{b,y}}(H_1,H_2)$ which asks that $\Delta ^{b}(x) = \frac{{\alpha_1}}{H_1}$. Therefore, such $x$ does not exist and the level set $E_{L_{{\alpha_1}}^b,L_{\alpha_2}^{b,y}}(H_1,H_2)$ is empty, leading to the conclusion that the value of the bivariate multifractal spectrum is  $-\infty$.

\bigskip

\textbf{Case 2}: ${H_2}>\frac{\alpha_2H_1}{\alpha_1}$. 

Suppose there exist infinitely many $n$ such that for all ${l\in \left[n, \left\lfloor n\left(\frac{\alpha_1}{H_1}-1\right)\right\rfloor+n\right]}$, the length of the block of 0's in the $b$-ary expansion of $y$ followed by $\theta_l$ is less than or equal to $\left\lfloor l\left(\frac{\alpha_2}{H_2}-1\right)\right\rfloor$. Then $n_{2k+1}$ can be taken as an infinite subsequence of these $n$'s. Hence, condition \ref{c1} is true. On the other hand, by $H_2>\frac{\alpha_2H_1}{\alpha_1}$, condition \ref{c2} holds automatically.  Thus, by the same proof as in Theorem \ref{th}, A Cantor type subset of $E_{L_{{\alpha_1}}^b,L_{\alpha_2}^{b,y}}(H_1,H_2)$ can be constructed, and it holds that
$
\mathcal{D}_{L_{{\alpha_1}}^b,L_{{\alpha_2}}^{b,y}}(H_1, H_2)=
    \min\left\{\frac{H_1}{{\alpha_1}}, \frac{H_2}{{\alpha_2}}\right\}.
$

 Otherwise, for any sufficiently large $n$, there exist $l_n \in \left[n, \left\lfloor n\left(\frac{\alpha_1}{H_1}-1\right)\right\rfloor+n\right]$, such that the length of the block of 0's followed by $\theta_{l_n}$ is larger than $\left\lfloor l_n\left(\frac{\alpha_2}{H_2}-1\right)\right\rfloor$. For any $x\in E_{L_{{\alpha_1}}^b,L_{\alpha_2}^{b,y}}(H_1,H_2)$, it follows that $\Delta^b(x)=\frac{\alpha_1}{H_1}$. Then, there exist infinitely many $n'$ such that the $b$-ary expansion of $x$ satisfies $$\varepsilon_{n'+1}\varepsilon_{n'+2}\cdots\varepsilon_{n'+\left\lfloor n'\left(\frac{\alpha_1}{H_1}-1\right)\right\rfloor}=0^{\left\lfloor n'\left(\frac{\alpha_1}{H_1}-1\right)\right\rfloor}.$$ 
Thus, for infinitely many $n'$,
\begin{equation} \nonumber
   \lVert b^{n'}(x-y) \rVert <  \frac{1}{b^{ \left\lfloor n'\left(\frac{{\alpha_2}}{H_2}-1\right)\right\rfloor}},
\end{equation}
which implies that 
\begin{equation}
    \nonumber
    \Delta^b(x-y)> \frac{\alpha_1}{H_1}.
\end{equation}
This contradicts the definition of the level set $E_{L_{{\alpha_1}}^b,L_{\alpha_2}^{b,y}}(H_1,H_2)$ which asks that $\Delta ^{b}(x-y) = \frac{{\alpha_2}}{H_2}$. Therefore, such $x$ does not exist and the level set $E_{L_{{\alpha_1}}^b,L_{\alpha_2}^{b,y}}(H_1,H_2)$ is empty, leading to the conclusion that the value of the bivariate multifractal spectrum is  $-\infty$.  
\end{proof}

%%%%%%%%%%%%%%%%%%%%%%%%%%%%%%%%%%%%%%%%%%%%%%%%%%%%%%%%%%%%%%%%%%%%%%%%%%%%%%%%%%%%%%%%%%%%%%%%%%%%%%%%%%%%%%%%%%%%%%%%%%%%%%%%%%%%%%%%%%%%%%%%%%%%%%%%%%%%%%%%%%%%%%%%%%%%%
\section{Proof of Propositions \ref{prop3} and \ref{prop1}} \label{section7}
%%%%%%%%%%%%%%%%%%%%%%%%%%%%%%%%%%%%%%%%%%%%%%%%%%%%%%%%%%%%%%%%%%%%%%%%%%%%%%%%%%%%%%%%%%%%%%%%%%%%%%%%%%%%%%%%%%%%%%%%%%%%%%%%%%%%%%%%%%%%%%%%%%%%%%%%%%%%%%%%%%%%%%%%%%%%%%
\begin{proof}[Proof of Proposition \ref{prop3}]
Let $l_{i+1}=\lfloor l_{i}\eta\rfloor$ and suppose $y$ has $b$-ary expansion $\theta_1\theta_2\theta_3\cdots$. We consider
\begin{equation} \label{levelset}
\begin{aligned}
E_{L_{{\alpha_1}}^b,L_{\alpha_2}^{b,y}}(H_1,H_2)&=\left\{x \in \mathbb{R}:\Delta^{b}(x) = \frac{{\alpha_1}}{H_1}\right\}\cap\left\{x \in \mathbb{R}: \Delta^{b}(x-y)=\frac{{\alpha_2}}{H_2}\right\}, 
\end{aligned}
\end{equation}
where $(H_1,H_2)\in[0,{\alpha_1}]\times [0,{\alpha_2}] \mathbin{\Big\backslash} K_{\alpha_1,\alpha_2}(y).$

We will show that if \begin{equation}
    \label{condition1}
    \frac{{\alpha_2}{H_1}^2}{{\alpha_1}^2}< H_2<\frac{{\alpha_2}}{{\alpha_1}}H_1,\ \text{and}\  
    \frac{\alpha_1}{\eta}<H_1\leqslant \alpha_1,
\end{equation}
then the conditions \ref{c1} and \ref{c2} are satisfied. Hence, by following the same lines of the proof of Theorem \ref{th}, a Cantor type set of $E_{L_{{\alpha_1}}^b,L_{\alpha_2}^{b,y}}(H_1,H_2)$ will be constructed and it follows that $
\mathcal{D}_{L_{{\alpha_1}}^b,L_{{\alpha_2}}^{b,y}}(H_1, H_2)=
    \min\left\{\frac{H_1}{{\alpha_1}}, \frac{H_2}{{\alpha_2}}\right\}.
$

% If $\Delta^b(x)=\frac{\alpha_1}{H_1}$, then there exist $x$ satisfying $\Delta^b(x-y)=\frac{\alpha_2}{H_2}$, since $H_2<\frac{{\alpha_2}}{{\alpha_1}}H_1$. Thus we only need to prove that, if $\Delta^b(x-y)=\frac{\alpha_2}{H_2}$, then there exist $x$ satisfying $\Delta^b(x-y)=\frac{\alpha_1}{H_1}$.

In fact, if \eqref{condition1} holds, since $\frac{{\alpha_2}{H_1}^2}{{\alpha_1}^2}< H_2$, then for all large $l_{i+1}$, it holds that
\begin{equation}
    \nonumber
    l_{i+1}\frac{H_1}{\alpha_1} < l_{i+1} \frac{\alpha_1H_2}{\alpha_2H_1}.
\end{equation}
and there exists at least an integer $n_i$ satisfying
\begin{equation}
    \label{n_kexist}
     l_{i+1}\frac{H_1}{\alpha_1} \leqslant n_i\leqslant l_{i+1} \frac{\alpha_1H_2}{\alpha_2H_1},
\end{equation}
i.e.,
\begin{equation}
\begin{aligned}
     \nonumber
    &l_{i+1}\leqslant n_i \frac{H_1}{{\alpha_1}},
    \\&n_i\leqslant l_{i+1}\frac{{\alpha_1}H_2}{{\alpha_2}H_1}.
\end{aligned}
\end{equation}
% By $\eta>\frac{\alpha_1}{H_1}$, we have
% \begin{equation}
%     \nonumber
%     n_k<l<\lfloor n_k\eta \rfloor,
% \end{equation}
% and
% \begin{equation} \nonumber
%   \lfloor n_k\eta \rfloor \leqslant l+\left\lfloor l\left(\frac{{\alpha_1}}{H_1}-1\right)\right\rfloor <\lfloor n_{k+1}\eta \rfloor.
% \end{equation}
Thus,
\begin{equation}
    \begin{aligned}
        \label{l_i}
        &l_{i+1}-n_i\leqslant \left\lfloor n_i\left(\frac{{\alpha_1}}{H_1}-1\right)\right\rfloor,\\&
        n_i+\left\lfloor n_i\left(\frac{{\alpha_2}}{H_2}-1\right)\right\rfloor-l_{i+1} \leqslant \left\lfloor l_{i+1}\left(\frac{{\alpha_1}}{H_1}-1\right)\right\rfloor.
    \end{aligned}
\end{equation}
Therefore, by \eqref{l_i}, there exists a subsequence of $\{n_{i}\}$ which is still be denoted as  $\{n_{i}\}$, such that $\frac{in_i}{n_{i+1}}\to 0, (i\to\infty),$ and for all $p\in \left[n_{i}, \left\lfloor n_{i}\left(\frac{\alpha_2}{H_2}-1\right)\right\rfloor+n_{i}\right]$, the length of the block of 0's followed by $\theta_p$ in the $b$-ary expansion of $y$ is less than or equal to $\left\lfloor l\left(\frac{\alpha_1}{H_1}-1\right)\right\rfloor$. Take $\{n_i\}$ as the sequence $\{n_{2k}\}$ in the proof of Theorem \ref{th}. Then condition \ref{c1} holds. On the other hand, by the assumption $H_2<\frac{\alpha_2H_1}{\alpha_1}$, condition \ref{c2} holds trivially. 
% By the choice of $n_{i}$, condition \ref{c1} holds.
% It follows that if $\Delta^b(x-y)=\frac{\alpha_2}{H_2}$, then there exist $x$ satisfying $\Delta^b(x-y)=\frac{\alpha_1}{H_1}$.  
Then by taking the same lines of the proof of Theorem \ref{th}, we can conclude that $
\mathcal{D}_{L_{{\alpha_1}}^b,L_{{\alpha_2}}^{b,y}}(H_1, H_2)=
    \min\left\{\frac{H_1}{{\alpha_1}}, \frac{H_2}{{\alpha_2}}\right\}.
$

Now, we will show that if
\begin{equation}
    \nonumber
    0\leqslant H_2\leqslant \frac{{\alpha_2}{H_1}^2}{{\alpha_1}^2}\  \text{and}\  \frac{\alpha_1}{\eta}<H_1\leqslant \alpha_1,
\end{equation}
then the condition \ref{c2} can never be satisfied. It then follows that $E_{L_{{\alpha_1}}^b,L_{{\alpha_2}}^{b,y}}(H_1, H_2)=\emptyset $. 
% for any $x$ satisfying  $\Delta^b(x-y)=\frac{\alpha_2}{H_2}$, we have $\Delta^b(x-y)>\frac{\alpha_1}{H_1}$. Which contradicts the condition of the level set that $\Delta ^{b}(x) = \frac{{\alpha_1}}{H_1}$. Thus, the level set is empty, leading to the conclusion that the value of the bivariate multifractal spectrum is  $-\infty$.   

In fact, for any $i$ and $p\in[l_i, l_{i+1})$, if $l_i\leqslant p\leqslant \frac{l_{i+1}H_1}{\alpha_1},$ it holds that
\begin{equation} \label{condition3}
l_{i+1}-p \geqslant \left\lfloor p\left(\frac{{\alpha_1}}{H_1}-1\right)\right\rfloor.    
\end{equation}
If $p> \frac{l_{k+1}H_1}{\alpha_1},$ then by $H_2\leqslant \frac{{\alpha_2}{H_1}^2}{{\alpha_1}^2}$, it follows that
\begin{equation}
\nonumber
    p > \frac{l_{k+1}\alpha_1H_2}{\alpha_2H_1}.
\end{equation}
Hence
\begin{equation}     \label{condition4}
   p+\left\lfloor p\left(\frac{{\alpha_2}}{H_2}-1\right)\right\rfloor-l_{k+1} > \left\lfloor l_{k+1}\left(\frac{{\alpha_1}}{H_1}-1\right)\right\rfloor. 
\end{equation}
By \eqref{condition3} and \eqref{condition4}, for any sufficiently large $n$, there exist $p_n\in \left[n, \left\lfloor n\left(\frac{\alpha_2}{H_2}-1\right)\right\rfloor+n\right]$, such that the length of the block of 0's in the $b$-ary expansion of $y$ followed by $\theta_{p_n}$ is larger than $\left\lfloor p_n\left(\frac{\alpha_1}{H_1}-1\right)\right\rfloor$. Note that for any $x\in E_{L_{{\alpha_1}}^b,L_{\alpha_2}^{b,y}}(H_1,H_2)$, it holds that $\Delta^b(x-y)=\frac{\alpha_2}{H_2}$. Then there exist infinitely many $n'$ such that the $b$-ary expansion $\varepsilon_1\varepsilon_2\cdots$ of $x$ satisfies $$\varepsilon_{n'+1}\varepsilon_{n'+2}\cdots\varepsilon_{n'+\left\lfloor n'\left(\frac{\alpha_2}{H_2}-1\right)\right\rfloor}=\theta_{n'+1}\theta_{n'+2}\cdots\theta_{n'+\left\lfloor n'\left(\frac{\alpha_2}{H_2}-1\right)\right\rfloor}.$$  Hence, for infinitely many $n'$, it follows that
\begin{equation} \nonumber
   \lVert b^{n'}x \rVert <  \frac{1}{b^{ \left\lfloor n\left(\frac{{\alpha_1}}{H_1}-1\right)\right\rfloor}}.
\end{equation}
This  implies
\begin{equation}
    \nonumber
    \Delta^b(x)> \frac{\alpha_1}{H_1},
\end{equation}
which contradicts the definition of the level set asking that $\Delta ^{b}(x) = \frac{{\alpha_1}}{H_1}$. Therefore, the level set is empty.

Similarly, if $\frac{{\alpha_1}{H_2}^2}{{\alpha_2}^2}< H_1 <\frac{{\alpha_1}}{{\alpha_2}}H_2,$ and $\frac{\alpha_2}{\eta}<H_2\leqslant \alpha_2$, then
\begin{equation*}
    \mathcal{D}_{L_{{\alpha_1}}^b,L_{\alpha_2}^{b,y}}(H_1, H_2)=\min\left\{\frac{H_1}{{\alpha_1}}, \frac{H_2}{{\alpha_2}}\right\}.
\end{equation*}
If $0\leqslant H_1\leqslant\frac{{\alpha_1}{H_2}^2}{{\alpha_2}^2}$ and $\frac{\alpha_2}{\eta}<H_2\leqslant \alpha_2$, then $E_{L_{{\alpha_1}}^b,L_{\alpha_2}^{b,y}}(H_1,H_2)=\emptyset$.
 \end{proof}

\begin{proof}[Proof of Proposition \ref{prop1}]
We consider the level set $E_{L_{\alpha_1}^b,L_{\alpha_2}^{b,y}}(H_1,H_2)$ for any  $(H_1 , H_2)\in  [0,{\alpha_1}]\times \big[0,{\alpha_2}]$. Let $n_i=\lfloor l_i\eta\rfloor$, for any $i$. The length of the block of 0's followed by $\theta_{n_i}$ in the $b$-ary expansion of $y$ is 1. Then, a subsequence  can be extracted from $\{n_i\}$ containing infinitely many terms which is still denoted as $\{n_{i}\}$ satisfying
$$
    \frac{i\cdot n_i}{n_{i+1}}\to 0, (i\to\infty),
$$
and for each $p\in \left[n_i, n_i+\max \left\{ \left\lfloor {n_i}\left(\frac{{\alpha_1}}{H_1}-1\right)\right\rfloor,\left\lfloor {n_i}\left(\frac{{\alpha_2}}{H_2}-1\right)\right\rfloor\right\} \right]$,
$$
       \|b^{p}y\|> \frac{1}{b^{\min\left\{ \left\lfloor {p}\left(\frac{{\alpha_1}}{H_1}-1\right)\right\rfloor,\left\lfloor {p}\left(\frac{{\alpha_2}}{H_2}-1\right)\right\rfloor\right\}} }.
$$
Hence, conditions \ref{c1} and \ref{c2} are satisfied. Then, by taking the same proof as that of Theorem \ref{th}, we can construct a Cantor type subset of $E_{L_{{\alpha_1}}^b,L_{\alpha_2}^{b,y}}(H_1,H_2)$ and deduce that
$$
\mathcal{D}_{L_{{\alpha_1}}^b,L_{{\alpha_2}}^{b,y}}(H_1, H_2)=
    \min\left\{\frac{H_1}{{\alpha_1}}, \frac{H_2}{{\alpha_2}}\right\}.
$$
\end{proof}

\section{Conclusion}
%%%%%%%%%%%%%%%%%%%%%%%%%%%%%%%%%%%%%%%%%%%%%%%%%%%%%%%%%%%%%%%%%%%%%%%%%%%%%%%%%%%%%%%%%%%%%%%%%%

Let us first discuss the choice of the H\"older exponent for the analysis of L\'evy functions. This is, in fact, motivated by the fact that if $ \alpha >0$ then  the function $L^b_\alpha$ is locally bounded, so that there is no need to have recourse to the regularity exponent fitted to a more general setting, such as $p$-exponents, which only require that $f$ is locally $p$-integrable (i.e., $f \in L^p_{loc}$).   However, the question can be raised if we allow to take negative values for $\alpha$. Indeed, in this case, the series  still converges in the sense of distributions (see \cite{jaffard2016p}, or, more simply, remark that a formal term by term integration  of sufficiently large order yields a normally convergent series), and one can wonder if a multifractal analysis of these generalized L\'evy functions can be performed using $p$-exponents. 

Let us  focus on the case $b=2$. Indeed, in this case, one can simply compute the coefficients of the L\'evy functions on the Haar basis, which allows to determine for which values of $p$ it belongs $L^p_{loc} $.
Let us first recall the definition of Haar basis. 
Let $ \varphi = 1_{[0,1]} $ 
and  let 
\begin{equation} \label{defhaa}  
  \psi (x) = \varphi  (2 x)- \varphi  (2 x-1), \quad \forall x \in \mathbb{R}.\end{equation} 
The  Haar basis on $[0,1]$  is  the orthonormal basis of $L^2( [0,1]
 ) $ composed of the functions
\begin{equation} \label{defbasond} 
\left\{ 
\begin{array}{ll}
   \varphi (x ) &        \\
 2^{j/2} \psi (2^jx
-k)  & \mbox{ for  }  j\geqslant0 \mbox{ and  }  k\in \{ 0, \cdots 2^j-1\}  . 
\end{array}
\right.
\end{equation} 
We denote by $c_{j,k}$ the coefficients of a function $f$ on this orthonormal basis. Let us compute the coefficients of the dyadic L\'evy functions on the Haar basis.
Let $$g_l (x) = \{ 2^l x \}, \quad \forall x \in \mathbb{R}.$$ If $j <l$, then the average of $g_l$ on dyadic intervals  of length at least $2^{-l +1}$ vanishes. It follows that  $< \psi_{j,k} | g_l >=0$.  If $j \geqslant l$, then  $g_l$ is an affine function of slope $2^l$ on the support of $\psi_{j,k} $, and an immediate computation yields 
\begin{equation}
    \nonumber
    < \psi_{j,k} | g_l> = -\frac{1}{4} 2^{l-j} . 
\end{equation}
It follows that the wavelet coefficients of $L_{\alpha}^2$ are 
\begin{equation}
    \nonumber
    c_{j,k}  =  -\frac{1}{4} 2^{-j/2} \sum_{l=0}^j 2^{(1-\alpha ) l } , 
\end{equation}
 for which, assuming $\alpha <1$, we can write 
 \begin{equation}
     \nonumber
      c_{j,k} = C_1 2^{-\alpha j }   + C_2 2^{- j },\quad \text{with}\quad C_1, C_2\ \text{being two constants}.
 \end{equation}
Recall that, for $p \in (1, \infty)$, a norm equivalent to the  $L^p$ norm of  a function $f$ is given by 
\begin{equation} \label{lpnorm}
    \parallel f \parallel^p_{L^p} = \left\lVert \left( \sum_j \sum_k | c_{j,k} |^2 2^j \chi_\lambda \right)^{1/2}  \right\lVert^p_{L^p},
\end{equation}
where $\lambda$ denotes the dyadic interval $\lambda = [ k 2^{-j},\; (k +1)2^{-j} ] $. 
For $f =L_{\alpha}^2$,  the double sum in the parenthesis in \eqref{lpnorm} boils down to 
\[ \sum_j 2^{-2\alpha j} ,\]
which diverges if $\alpha <0$.  We conclude that using $p$-exponents does not allow to deal with a larger range of values of $\alpha$; indeed, as soon as $\alpha <0 $, the L\'evy functions $L_{\alpha}^2$ are Schwartz distributions which do not belong to any $L^p$ space.  

\medskip 

The entire bivariate multifractal spectrum of L\'evy function $L_{\alpha_1}^b$ and translated L\'evy function $L_{\alpha_2}^{b,y}$ for almost every $y$ was determined, that is,
 \begin{equation}
    \begin{aligned}
        \nonumber
 \mathcal{D}_{L_{{\alpha_1}}^b,L_{{\alpha_2}}^{b,y}}(H_1, H_2)=
  \begin{cases}
    \min\left\{\frac{H_1}{{\alpha_1}}, \frac{H_2}{{\alpha_2}}\right\},       & \quad {\rm if}\ (H_1 , H_2)\in  [0,{\alpha_1}]\times \big[0,{\alpha_2}],\\
-\infty       & \quad {\rm else}.
  \end{cases}
    \end{aligned}
\end{equation}

For  other values of $y$, the spectrum in the region $K_{\alpha_1,\alpha_2}(y)$ was determined. For values outside $K_{\alpha_1,\alpha_2}(y)$, we were able to prove that the spectrum is either the minimum of the spectra of the two L\'evy functions respectively or $-\infty$, this leaves open the determination of the bivariate multifractal spectrum  on the region $ [0,{\alpha_1}]\times [0,{\alpha_2}] \setminus K_{ \alpha_1,\alpha_2 }(y)$. 

\smallskip

From the multifractal analysis of L\'evy functions, it appears that the spectrum is consistently either the empty set or the formula of large intersection as given by \eqref{minv}. This result allows to  conjecture that, under mild hypotheses,  
 the bivariate multifractal spectra of pairs of jump functions exhibit a distinct structure: the spectrum at the values given by the bivariate level sets is either $-\infty$ or corresponds to the minimum of the individual spectra of the two functions.  
 We will come back to this question in  \cite{Qian3}.

\medskip

Another important direction of research is to extend the present work to the bivariate multifractal analysis of L\'evy functions $L_{\alpha_1}^{b_1}$ and $L_{\alpha_2}^{b_2}$ with different bases $b_1$ and $b_2$ respectively, where $b_1$ and $b_2$ are coprime. We expect that this would yield  further insights into how varying the base parameters affects the multifractal properties of L\'evy functions. By investigating functions with different bases, we aim to develop a more generalized framework for multifractal spectrum, potentially uncovering new phenomena  that could be applicable to a broader class of functions. This work would not only deepen our theoretical understanding but also pave the way for practical applications in fields where multifractal analysis is utilized, and in particular for applications where possible shifts of (known or unknown quantities) may occur between the data under analysis.

Moreover, we are interested in Hecke's functions  \[ H_{\alpha}(x)=\frac{\{nx\}}{n^{\alpha}} . \]   A next interesting challenge is to  perform the multivariate multifractal analysis of Hecke's function $H_{\alpha_1}$ and the translated Hecke's function $H_{\alpha_2}^y$ which satisfies $H_{\alpha_2}^y=H_{\alpha_2}(\cdot-y)$ for a given translation parameter $y\in \mathbb{R}$, which means that we aim at obtaining the bivariate multifractal spectrum of Hecke's translated functions. 

 \section*{Acknowledgements}
Lingmin LIAO was partially supported by the Fundamental Research Funds for the Central Universities of China(1301/600460054).
\bibliographystyle{plain}
\bibliography{references}
\end{document}